\documentclass[11pt]{article}

\usepackage[T1]{fontenc}
\usepackage[utf8]{inputenc}
\usepackage[margin=1in]{geometry}
\usepackage{mathpazo}
\usepackage{courier}
\usepackage{microtype}
\usepackage{mathtools}
\usepackage{amssymb}
\usepackage{amsthm}
\usepackage{tikz-cd}
\usepackage{enumitem}
\usepackage{xcolor}
\usepackage[numbers,sort&compress]{natbib}
\usepackage{authblk}
\usepackage{hyperref}
\usepackage[nameinlink,noabbrev]{cleveref}

\theoremstyle{plain}
\newtheorem{theorem}{Theorem}[section]
\newtheorem{proposition}[theorem]{Proposition}
\newtheorem{lemma}[theorem]{Lemma}
\newtheorem{corollary}[theorem]{Corollary}
\newtheorem*{maintheorem}{Main Theorem}
\theoremstyle{definition}
\newtheorem{definition}[theorem]{Definition}
\newtheorem{example}[theorem]{Example}

\theoremstyle{remark}
\newtheorem{remark}[theorem]{Remark}

\crefname{theorem}{Theorem}{Theorems}
\crefname{proposition}{Proposition}{Propositions}
\crefname{lemma}{Lemma}{Lemmas}
\crefname{corollary}{Corollary}{Corollaries}
\crefname{definition}{Definition}{Definitions}
\crefname{example}{Example}{Examples}
\crefname{remark}{Remark}{Remarks}
\crefname{notation}{Notation}{Notations}
\crefname{equation}{Equation}{Equations}
\crefname{section}{Section}{Sections}
\Crefname{equation}{Equation}{Equations}
\Crefname{section}{Section}{Sections}

\newcommand{\CAT}{\mathbf{CAT}}
\newcommand{\Cat}{\mathbf{Cat}}
\newcommand{\SCoFib}{\mathbf{SCoFib}}
\newcommand{\ClCoFib}{\mathbf{ClCoFib}}
\newcommand{\NClCoFib}{\mathbf{NClCoFib}}
\newcommand{\SFib}{\mathbf{SFib}}
\newcommand{\Alg}{\mathbf{Alg}_{\mathrm{s}}}
\newcommand{\PsAlg}{\mathbf{PsAlg}}
\newcommand{\OFSch}{\mathbf{OFS}_{\mathrm{ch}}}
\newcommand{\CoCart}{\operatorname{CoCart}}
\newcommand{\Vrt}{\operatorname{Vert}}
\newcommand{\Iso}{\operatorname{Iso}}
\newcommand{\Dom}{\operatorname{Dom}}
\newcommand{\Cod}{\operatorname{Cod}}
\newcommand{\C}{\mathbb{C}}
\newcommand{\Sq}{\mathbb{S}}
\newcommand{\Fact}{\mathsf{Fact}}
\numberwithin{equation}{section}
\title{From Grothendieck cofibrations to factorization systems: \\a formal 2-monadic account}
\author[1]{Fernando Lucatelli Nunes}
\author[2]{Walter Tholen}
\affil[1]{CMUC, Department of Mathematics, University of Coimbra,
3000-143 Coimbra, Portugal}
\affil[2]{Department of Mathematics and Statistics, York University,
4700 Keele Street, Toronto, Ontario M3J 1P3, Canada}
\date{}

\begin{document}
\maketitle

\begin{abstract}
Grothendieck cofibrations describe transport in a category varying over
a base, while factorization systems organize the arrows of a category
into two complementary classes.  We give a fully 2-categorical account
of the passage from the former structure to the latter.  The global
comma 2-monad on the arrow 2-category encodes Grothendieck transport,
whereas the squaring 2-monad encodes factorizations.  We prove that
split cofibrations are precisely the strict algebras for the comma
2-monad, including their 1-cells and 2-cells, and that normally cloven
cofibrations are precisely its normal pseudoalgebras.  A canonical
colax morphism from the comma 2-monad to the squaring 2-monad then
turns cocartesian transport into the cocartesian--vertical
factorization of arrows in the total category.  At the strict level,
this yields the strict factorization system of designated
cocartesian and vertical arrows; at the coherent level, it yields the
orthogonal factorization system whose left class consists of all
cocartesian arrows and whose right class consists of the arrows sent
to isomorphisms in the base.  We also separate unrestricted global
pseudoalgebras, which retain a coherently trivial base action, from
fixed-base pseudoalgebras, which correspond to arbitrary cleavages,
and record the dual strict result for fibrations.  This places the
classical cofibration--factorization interaction, in all these
variants, within a single change-of-2-monads construction and relates
it directly to the existing fibrational and factorization literature.
\end{abstract}



\section{Introduction}

We begin with a conceptual and then a precise account of the results
proved in this paper.  In the second part of the introduction, we place
them in the extensive literature on Grothendieck fibrations,
2-dimensional monad theory, and factorization systems.

A Grothendieck cofibration $P\colon E\to B$ may be regarded as a
category of objects and arrows varying over the base $B$.  An arrow of
the base prescribes a direction in which an object of one fibre may be
transported to another fibre.  A cleavage makes this transport
explicit by choosing a cocartesian lift, and a split cleavage makes
identity and composite transport strictly functorial.  This simple
idea is one of the standard categorical forms of substitution and
change of parameters: the fibres retain the local data, while the
cocartesian arrows describe how those data move.
The dual fibrational viewpoint is the organizing perspective of
categorical logic and type theory \citep[Chapter~1]{Jacobs1999}.

A factorization system organizes a different kind of structure.
Instead of transporting objects between fibres, it decomposes every
arrow of a single category into two successive stages, governed by a
unique lifting property.  For the total category $E$ of a
cofibration, a cleavage already suggests such a decomposition: first
transport the domain along the image of the arrow in $B$, and then
complete the arrow inside the resulting fibre.  The existence of this
factorization is elementary.  What is less visible is why its strict
and coherent forms, its functoriality under varying base categories,
and the resulting orthogonality should all be controlled by the same
piece of structure.

The answer is that transport and factorization are algebras for two
different 2-monads.  The comma 2-monad remembers successive movement
in the base; the squaring 2-monad remembers the middle object of a
factorization.  The comparison constructed below is a colax morphism
between these 2-monads.  Restriction of scalars along it converts the
transport action into the factorization action.  In this way, the
familiar cocartesian--vertical decomposition is not an additional
choice placed on the total category: it is the image of the original
Grothendieck transport under a canonical change of 2-monads.

Let $P\colon E\to B$ be a Grothendieck cofibration.  A cleavage
transports $x\in E$ along $u\colon Px\to b$ by a chosen cocartesian
arrow
\begin{equation*}
  \delta^x_u\colon x\longrightarrow u_!x.
\end{equation*}
The same arrow determines, for every $f\colon x\to y$, a unique
vertical arrow $\nu_f$ and hence a factorization
\begin{equation}
\begin{tikzcd}[column sep=huge]
x \arrow[r,"\delta^{x}_{Pf}"]
  & (Pf)_{!}x \arrow[r,"\nu_f"]
  & y .
\end{tikzcd}
\qquad f=\nu_f\delta^{x}_{Pf}.
\label{eq:intro-factorization}
\end{equation}
The aim of this paper is to formalize the passage from the transport
$\delta$ to the factorization in \Cref{eq:intro-factorization}.  The
point is not merely that the factorization exists or is functorial.  Its
entire algebraic structure is obtained from the transport action by a
change of 2-monads.

The first 2-monad is the global comma 2-monad $\C$ on the strict arrow
2-category $\CAT^{\mathbf 2}$.  At $P\colon E\to B$ it is
\begin{equation*}
  \C P=\Cod_{P}\colon
  P\!\downarrow\!B\longrightarrow B.
\end{equation*}
Its unit inserts identity base arrows and its multiplication composes
successive base arrows.  The second is the squaring 2-monad $\Sq$ on
$\CAT$, given by $\Sq E=E^{\mathbf 2}$.  We construct a 2-natural
transformation
\begin{equation*}
  \kappa\colon\Sq\Dom\Longrightarrow\Dom\C,
  \qquad
  \kappa_P(f\colon x\to y)=(x,Pf),
\end{equation*}
and prove that $(\Dom,\kappa)$ is a colax morphism of 2-monads.  Thus
restriction of scalars turns a comma-monad action into a squaring-monad
action.  Applied to the action determined by a cleavage, this is exactly
\Cref{eq:intro-factorization}.

There are two levels of this statement.  Write $\SCoFib$ for split
cofibrations and strictly cleavage-preserving squares, and write
$\NClCoFib$ for normally cloven cofibrations and squares whose total
functors preserve cocartesian arrows; in both cases all compatible
arrow-2-category 2-cells are retained.  The two rows of
\begin{equation}
\begin{tikzcd}[column sep=large,row sep=large]
\SCoFib
  \arrow[r,"\Gamma", "\cong" description]
  \arrow[d,hook]
& \Alg(\C)
  \arrow[r,"\kappa^{*}"]
  \arrow[d,hook]
& \Alg(\Sq)
  \arrow[d,hook] \\
\NClCoFib
  \arrow[r,"\Gamma_{\mathrm n}", "\cong" description]
& \PsAlg_{\mathrm n}(\C)
  \arrow[r,"\kappa^{*}_{\mathrm{ps}}"']
& \PsAlg_{\mathrm n}(\Sq)
\end{tikzcd}
\label{eq:intro-architecture}
\end{equation}
commute.  Here $\Alg$ denotes strict algebras and strict morphisms,
whereas $\PsAlg_{\mathrm n}$ denotes normal pseudoalgebras,
pseudomorphisms, and algebra transformations.  By
\citet[Theorem~2.4]{KorostenskiTholen1993}, the lower-right term is
isomorphic to
$\OFSch$, the 2-category of orthogonal factorization systems with normalized
chosen factorizations and class-preserving functors.  Consequently,
\Cref{eq:intro-architecture} records both the strict and coherent
passages from Grothendieck transport to factorization.

\begin{maintheorem}[The 2-monadic passage from transport to
factorization]
There are isomorphisms of 2-categories over $\CAT^{\mathbf 2}$
\begin{equation*}
  \SCoFib\cong\Alg(\C),
  \qquad
  \NClCoFib\cong\PsAlg_{\mathrm n}(\C).
\end{equation*}
Restriction of scalars along the colax monad morphism
$(\Dom,\kappa)$ gives the two horizontal
2-functors in \Cref{eq:intro-architecture}.  For every normally cloven
cofibration $P\colon E\to B$, the resulting normal pseudo-$\Sq$-algebra
selects the factorization in \Cref{eq:intro-factorization} and determines
the orthogonal factorization system
\begin{equation*}
  \bigl(\CoCart(P),P^{-1}\Iso(B)\bigr).
\end{equation*}
If the cleavage is split, the action is strict and its strict
factorization system is
\begin{equation*}
  (\Delta_P,\Vrt(P)),
\end{equation*}
where $\Delta_P$ is the class of designated cocartesian arrows.
Over a fixed base, the same restriction applies to every cleavage: the
orthogonal classes remain
$(\CoCart(P),P^{-1}\Iso(B))$, while the pseudo-$\Sq$-unit records the
possibly non-normalized factorization of each identity.
\end{maintheorem}

The strict comparison is proved componentwise.  A strict action
$A\colon P\!\downarrow\!B\to E$ defines
\begin{equation*}
  u_!x=A(x,u),
  \qquad
  \delta^x_u=A(1_x,u).
\end{equation*}
The action unit gives the identity transport, while the action
multiplication gives composite transport.  A further consequence of the
multiplication law, the absorption identity, supplies the cocartesian
universal property.  Conversely, a split cleavage defines $A$ by
cocartesian factorization.  We prove that these assignments are inverse
on objects, that strict algebra morphisms are exactly the squares
preserving the designated lifts, and that every arrow-2-category 2-cell
between such squares is automatically an algebra 2-cell.

The pseudoalgebra calculation separates normality from strictness.  A
normal pseudoaction gives identity transport strictly and composite
transport through the canonical vertical isomorphisms
\begin{equation*}
  \chi^x_{v,u}\colon v_!(u_!x)\xrightarrow{\ \cong\ }(vu)_!x.
\end{equation*}
Their triangle and pentagon axioms follow from cocartesian uniqueness.
Pseudomorphisms are precisely the squares preserving cocartesian arrows;
at $(x,u)$, their comparison is the unique vertical isomorphism between
the codomains of the chosen lift $\delta^{Tx}_{Su}$ and the image
$T\delta^x_u$ that makes the lift triangle commute.  Thus normal
pseudoalgebras replace split transport by
coherently associative chosen transport, and strict factorization
systems by ordinary orthogonal factorization systems with chosen
factorizations.

Normality is essential in the global statement.  An unrestricted
pseudo-$\C$-action in $\CAT^{\mathbf 2}$ has a base component
$R\colon B\to B$ whose unit includes an isomorphism
$1_B\cong R$.  Its distinguished lift of $u\colon Px\to b$ therefore
lands over $Rb$, not strictly over $b$.  Thus unrestricted global
pseudoalgebras contain additional coherently trivial base data and are
not the same data as cloven Grothendieck cofibrations; the analogous
fibrational obstruction is noted in
\citet[Remark~2.7(b)]{EmmeneggerEtAl2024}.  Over the strict slice
$\CAT/B$, no base twisting is possible, and arbitrary pseudoalgebras
correspond to cloven cofibrations over $B$.  Restriction along $\kappa$
then gives the same orthogonal classes as in the normal case, but the
factorization of $1_x$ is the chosen lift over $1_{Px}$ followed by its
inverse; see \Cref{cor:fixed-base-pseudo-factorization}.

The comma-monad description has a substantial history.  Gray first
constructed the free split (co)fibration over a fixed base
\citep[Theorem~3.9]{Gray1966}, supplying the free-algebra part of the
classical monadic description.  Street formulated fibrations in a 2-category by
means of comma objects and an adjoint-to-the-unit criterion
\citep{Street1974}, and later placed both their algebra structures and
the Grothendieck construction
in the setting of lax-idempotent doctrines and bicategories
\citep{Street1980}.  We use ``cofibration'' in Grothendieck's sense,
also called opfibration; Street's variance conventions are transferred
by reversing the relevant 2-cells.  The Eilenberg--Moore formulation is
part of Street's formal theory of monads \citep{Street1972}, with the
2-categorical conventions of \citet{KellyStreet1974}.

The global monad itself is not new.
\citet[Proposition~16]{FauserVickers2014} construct the corresponding
2-monad on the arrow 2-category of a representable 2-category.
\citet[Proposition~17]{FauserVickers2014} identify pseudoalgebras
whose action has identity base component with pseudo-opfibrations, and
identify their normalized instances with opfibrations.  In $\CAT$
this is the monad $\C$ used here.
\citet[Theorem~8.1]{PeschkeTholen2020} obtain the free
split-cofibration 2-adjunction as part of a larger web of global
2-adjunctions.  The fixed-base strict and pseudo identifications,
including their morphisms, are recorded in fibrational variance by
\citet[Theorem~2.6]{EmmeneggerEtAl2024}.  The same variation with the
base underlies the indexed category of split (op)fibrations whose
effective descent morphisms are studied by
\citet{LucatelliNunesPrezadoSousa2023}.  Our first contribution is a
direct global proof of the complete strict comparison, including its 1-cells
and 2-cells, together with the parallel comparison for normal
pseudoalgebras in the
same conventions.

The main contribution is the right-hand passage in
\Cref{eq:intro-architecture}.  The same-2-category prototype for
pseudoalgebraic restriction of scalars is
\citet[Theorem~A.7]{PerroneTholen2022}.  Since our comparison lies over
$\Dom\colon\CAT^{\mathbf 2}\to\CAT$, we prove the required relative
form in \Cref{prop:pseudo-restriction-scalars}.  The identification of
normal pseudo-$\Sq$-algebras with orthogonal factorization systems is
given by \citet[Theorem~2.4]{KorostenskiTholen1993}; the strict and lax
variants are studied in
\citet{RosebrughWood2002Distributive,RosickyTholen2002}.  At the level of arrow
classes, the resulting system is the cofibrational dual of the
prefibrational factorization in
\citet[Section~3.7]{RosickyTholen2007}.  Our change
of 2-monads explains why that system is attached to a cofibration and
how its chosen factorizations vary under squares and 2-cells.

Comprehensive factorization instead factors a functor as an initial
functor followed by a discrete opfibration
\citep{StreetWalters1973}.  In the present construction $P$ remains
fixed and the arrows factored are those of its total category $E$.

The paper is organized as follows.  \Cref{sec:cofibrations} fixes the
2-categorical and fibrational conventions.  \Cref{sec:monad} constructs
the global comma 2-monad, its free-algebra adjunction, and its
lax-idempotence.  \Cref{sec:monadicity} proves strict 2-monadicity.
\Cref{sec:pseudoalgebras} identifies normal pseudoalgebras with normally
cloven cofibrations and analyzes the global base-twisting phenomenon.
\Cref{sec:squaring} constructs $\kappa$ and derives the strict and
pseudo factorization results.  \Cref{sec:dual} records the dual strict
theorem for fibrations.

\paragraph{Size convention}
We work relative to a fixed enlargement of Grothendieck universes.
The symbol $\CAT$ denotes the resulting 2-category of locally small
categories, functors, and natural transformations; all categories and
comma objects under discussion belong to the chosen enlargement.  No
argument depends on the particular universe convention.
We write composites from right to left, so that $gf$ means
$g\mathbin{\circ}f$.

\section{Split cofibrations and the arrow 2-category}
\label{sec:cofibrations}

\subsection{The arrow 2-category}

Let $\mathbf 2=(0\to1)$ be the walking arrow.  We use
\begin{equation*}
\CAT^{\mathbf 2}=[\mathbf 2,\CAT]
\end{equation*}
for the strict arrow 2-category.  Its objects are functors
$P\colon E\to B$.  A 1-cell $(T,S)\colon P\to Q$, where
$Q\colon F\to C$, is a strictly commutative square
\begin{equation*}
\begin{tikzcd}[column sep=large,row sep=large]
E \arrow[r,"T"] \arrow[d,"P"'] & F \arrow[d,"Q"] \\
B \arrow[r,"S"'] & C
\end{tikzcd}
\qquad\text{with}\qquad QT=SP.
\end{equation*}
A 2-cell
$(\beta,\alpha)\colon(T,S)\Rightarrow(T',S')$ is a pair of
natural transformations $\beta\colon T\Rightarrow T'$ and
$\alpha\colon S\Rightarrow S'$ satisfying
\begin{equation}
Q\beta=\alpha P.
  \label{eq:arrow-2-cell}
\end{equation}
Both compositions are inherited componentwise from $\CAT$.
These conventions follow the standard calculus reviewed by
\citet{KellyStreet1974}.

We write
\begin{equation*}
\Dom,\Cod\colon\CAT^{\mathbf 2}\longrightarrow\CAT
\end{equation*}
for evaluation at $0$ and $1$, respectively.  Thus $\Dom(P)=E$ and
$\Cod(P)=B$.

\subsection{Strict 2-monads and their algebras}

We recall the strict 2-dimensional algebra used in the paper.  This
also fixes what is meant by strict 2-monadicity.

\begin{definition}\label{def:strict-em}
Let $\mathcal K$ be a strict 2-category.  A \emph{strict 2-monad}
$\mathbb T=(\mathbb T,\eta,\mu)$ on $\mathcal K$ consists of a strict
2-functor $\mathbb T\colon\mathcal K\to\mathcal K$ and 2-natural
transformations
\begin{equation*}
  \eta\colon1_{\mathcal K}\Rightarrow\mathbb T,
  \qquad
  \mu\colon\mathbb T^2\Rightarrow\mathbb T
\end{equation*}
satisfying
\begin{equation*}
  \mu\,\mathbb T\eta=1_{\mathbb T}
  =\mu\,\eta_{\mathbb T},
  \qquad
  \mu\,\mathbb T\mu=\mu\,\mu_{\mathbb T}.
\end{equation*}
A \emph{strict $\mathbb T$-algebra} is an object $X$ equipped with a
1-cell $a\colon\mathbb T X\to X$ such that
\begin{equation*}
  a\eta_X=1_X,
  \qquad
  a\mu_X=a\mathbb T a.
\end{equation*}
A strict algebra morphism
$f\colon(X,a)\to(Y,b)$ is a 1-cell satisfying
$fa=b\mathbb T f$.  An algebra 2-cell
$\alpha\colon f\Rightarrow g$ is a 2-cell in $\mathcal K$ satisfying
\begin{equation*}
  \alpha a=b\mathbb T\alpha.
\end{equation*}
These data form the strict Eilenberg--Moore 2-category
$\Alg(\mathbb T)$.  Its forgetful 2-functor is
\begin{equation*}
  U_{\mathbb T}\colon\Alg(\mathbb T)\longrightarrow\mathcal K.
\end{equation*}
\end{definition}

\begin{definition}[Pseudoalgebras]\label{def:pseudoalgebras}
Our orientation agrees with the lax-algebra
convention in which the structural 2-cells point towards the strict
algebra equations.  A \emph{pseudo $\mathbb T$-algebra} is a 1-cell
$a\colon\mathbb TX\to X$ together with invertible 2-cells
\begin{equation}
  \iota\colon 1_X\Rightarrow a\eta_X,
  \qquad
  \chi\colon a\mathbb T a\Rightarrow a\mu_X,
  \label{eq:pseudoalgebra-cells}
\end{equation}
such that, using 2-naturality of $\eta$ and $\mu$, together with the
monad laws, to identify parallel 1-cells,
\begin{align}
  (\chi\,\mathbb T\eta_X)(a\,\mathbb T\iota)&=1_a,
  &
  (\chi\,\eta_{\mathbb T X})(\iota\,a)&=1_a,
  \label{eq:pseudoalgebra-unit-axioms}\\
  (\chi\,\mathbb T\mu_X)(a\,\mathbb T\chi)
  &=(\chi\,\mu_{\mathbb T X})(\chi\,\mathbb T^2a).
  \label{eq:pseudoalgebra-associativity-axiom}
\end{align}
The first line consists of the two unit axioms and the second line is
the associativity axiom.  The pseudoalgebra is \emph{normal} when
$\iota$ is an identity.

A pseudomorphism
$(X,a,\iota^a,\chi^a)\to(Y,b,\iota^b,\chi^b)$ is a 1-cell
$f\colon X\to Y$ equipped
with an invertible 2-cell
\begin{equation}
  \overline f\colon b\mathbb T f\Rightarrow fa
  \label{eq:pseudomorphism-cell}
\end{equation}
satisfying
\begin{align}
  (\overline f\,\eta_X)(\iota^b f)
  &=f\iota^a,
  \label{eq:pseudomorphism-unit-axiom}\\
  (\overline f\,\mu_X)(\chi^b\,\mathbb T^2f)
  &=(f\chi^a)(\overline f\,\mathbb T a)
    (b\,\mathbb T\overline f).
  \label{eq:pseudomorphism-multiplication-axiom}
\end{align}
If $(f,\overline f)$ and
$(g,\overline g)$ are pseudomorphisms, an \emph{algebra
transformation} $\alpha\colon f\Rightarrow g$ is a 2-cell satisfying
\begin{equation}
  (\alpha a)\,\overline f
  =\overline g\,(b\mathbb T\alpha).
  \label{eq:pseudoalgebra-transformation-general}
\end{equation}
A \emph{strict morphism of pseudoalgebras} is a pseudomorphism whose
comparison cell $\overline f$ is an identity; in particular, its
underlying 1-cell satisfies $b\mathbb T f=fa$ strictly.
We denote the resulting 2-category by
$\PsAlg(\mathbb T)$ and its sub-2-category of normal pseudoalgebras by
$\PsAlg_{\mathrm n}(\mathbb T)$.
\end{definition}

In \Cref{sec:pseudoalgebras}, the axioms of
\Cref{def:pseudoalgebras} are displayed in their concrete transport
form.  The general definitions and pasting conventions agree with
\citet{BlackwellKellyPower1989}.

If $F\colon\mathcal K\rightleftarrows\mathcal A\colon U$ is a
2-adjunction $F\dashv U$ with
induced 2-monad $\mathbb T$, its comparison 2-functor sends an object
$A$ to the strict action $U\varepsilon_A\colon\mathbb T UA\to UA$.
We call $U$ \emph{strictly 2-monadic} when this comparison is an
isomorphism of 2-categories over $\mathcal K$.  Thus our use of
``strictly'' refers both to the algebraic cells in
\Cref{def:strict-em} and to the comparison itself.  General background
on this formalism may be found in
\citet{Street1972,BlackwellKellyPower1989}.

\subsection{Comma categories, fibres, and strict slices}

Let $P\colon E\to B$.  The comma category
$P\!\downarrow\!B$ has objects
\begin{equation*}
  (x,u),
  \qquad x\in E,\qquad u\colon Px\longrightarrow b\text{ in }B,
\end{equation*}
and an arrow
\begin{equation*}
  (f,v)\colon(x,u)\longrightarrow(y,w)
\end{equation*}
consists of $f\colon x\to y$ in $E$ and $v\colon b\to c$ in $B$
such that
\begin{equation}
\begin{tikzcd}[column sep=large,row sep=large]
Px \arrow[r,"u"] \arrow[d,"Pf"']
  & b \arrow[d,"v"] \\
Py \arrow[r,"w"'] & c .
\end{tikzcd}
\qquad\text{that is,}\qquad
vu=wPf.
  \label{eq:comma-arrow}
\end{equation}
Its codomain projection is
\begin{equation*}
  \Cod_{P}\colon P\!\downarrow\!B\longrightarrow B,
  \qquad (x,u)\longmapsto b,\qquad(f,v)\longmapsto v.
\end{equation*}
There is a canonical functor over $B$,
\begin{equation}
  H^{P}\colon E\longrightarrow P\!\downarrow\!B,
  \qquad
  H^{P}x=(x,1_{Px}),\qquad H^{P}f=(f,Pf).
  \label{eq:HP-preliminary}
\end{equation}

For $a\in B$, the \emph{fibre} $E_a$ has as objects those $x$ with
$Px=a$ and as arrows those $f$ with $Pf=1_a$.  The strict slice
2-category $\CAT/B$ has functors into $B$ as objects, functors over
$B$ as 1-cells, and natural transformations over the identity of $B$
as 2-cells.  Thus an adjunction in $\CAT/B$ is an ordinary adjunction
whose functors and unit and counit all lie strictly over $B$.

\subsection{Cocartesian arrows and splittings}

Let $P\colon E\to B$.  An arrow $e\colon x\to y$ of $E$ is
$P$-\emph{cocartesian} if, whenever $g\colon x\to z$ in $E$ and
$v\colon Py\to Pz$ in $B$ satisfy
\begin{equation*}
Pg=v\,Pe,
\end{equation*}
there is a unique $h\colon y\to z$ such that
\begin{equation}
he=g,
  \qquad
  Ph=v.
  \label{eq:cocartesian-up}
\end{equation}
We denote the class of such arrows by $\CoCart(P)$.  An arrow $f$ is
$P$-\emph{vertical} if $Pf$ is an identity; the wide subcategory of
vertical arrows is denoted by $\Vrt(P)$.

The functor $P$ is a \emph{cofibration} if, for every $x\in E$ and
every $u\colon Px\to b$ in $B$, there is a $P$-cocartesian arrow over
$u$ with domain $x$.  We use \emph{cofibration} in Grothendieck's
sense \citep[Expos\'e~VI]{Grothendieck1971}; it is also widely called
an opfibration.

A \emph{cleavage} chooses one such arrow for every pair $(x,u)$:
\begin{equation}
\delta^{x}_{u}\colon x\longrightarrow u_{!}x,
  \qquad
  P\delta^{x}_{u}=u.
  \label{eq:cleavage}
\end{equation}
Equivalently, every designated lifting is displayed by
\begin{equation*}
\begin{tikzcd}[column sep=large,row sep=large]
x \arrow[r,"\delta^{x}_{u}"] \arrow[d,mapsto]
  & u_{!}x \arrow[d,mapsto] \\
Px \arrow[r,"u"'] & b .
\end{tikzcd}
\end{equation*}
For $u\colon a\to b$, cocartesian uniqueness defines a functor
$u_{!}\colon E_a\to E_b$ on the fibres.  On a vertical arrow
$f\colon x\to y$, the arrow $u_{!}f$ is the unique vertical arrow
satisfying
\begin{equation}
(u_{!}f)\delta^{x}_{u}=\delta^{y}_{u}f.
  \label{eq:pushforward-arrow}
\end{equation}
The cleavage is \emph{normal} if
$\delta^{x}_{1_{Px}}=1_x$ (and hence $(1_a)_!=1_{E_a}$).  It is
\emph{split} if it is normal and the following composition equations hold
whenever the expressions are defined:
\begin{align}
  (vu)_{!}&=v_{!}u_{!},
  \label{eq:split-functors}\\
  \delta^{x}_{vu}&=\delta^{u_{!}x}_{v}\,\delta^{x}_{u}.
  \label{eq:split-delta}
\end{align}
A cofibration equipped with a split cleavage is a \emph{split
cofibration}.

\begin{example}\label{ex:product-cofibration}
For categories $B$ and $K$, the first projection
$\pi_1\colon B\times K\to B$ is a split cofibration.  At
$(a,k)$, the designated lift of $u\colon a\to b$ is
\begin{equation*}
  (u,1_k)\colon(a,k)\longrightarrow(b,k).
\end{equation*}
An arrow $(u,h)\colon(a,k)\to(b,\ell)$ then has the chosen
cocartesian--vertical factorization
\begin{equation*}
\begin{tikzcd}[column sep=large]
(a,k) \arrow[r,"{(u,1_k)}"]
  & (b,k) \arrow[r,"{(1_b,h)}"]
  & (b,\ell).
\end{tikzcd}
\end{equation*}
The construction studied in \Cref{sec:squaring} shows that this
elementary factorization, and its analogue for every split
cofibration, is obtained functorially by changing 2-monads.
\end{example}

Every arrow $f\colon x\to y$ in a cloven cofibration has a unique
factorization
\begin{equation}
f=\nu_f\,\delta^{x}_{Pf},
  \qquad
  P\nu_f=1_{Py}.
  \label{eq:cofib-factorization}
\end{equation}
Indeed, this is \Cref{eq:cocartesian-up} with $v=1_{Py}$.

\begin{definition}\label{def:scofib}
The 2-category $\SCoFib$ is defined as follows.
\begin{enumerate}[label=\textup{(\roman*)},leftmargin=2.2em]
\item Its objects are split cofibrations with their designated split
      cleavages.
\item A 1-cell $(T,S)\colon P\to Q$ belongs to $\SCoFib$ if it
      preserves the designated liftings strictly:
      \begin{equation}
        T\delta^{x}_{u}=\delta^{Tx}_{Su}
        \label{eq:cleavage-preservation}
\end{equation}
      for every $x\in E$ and $u\colon Px\to b$.  In particular,
      $T(u_{!}x)=(Su)_{!}Tx$ as an equality of objects.
\item Between such 1-cells, all 2-cells of $\CAT^{\mathbf 2}$ are
      admitted.
\end{enumerate}
The evident forgetful 2-functor
$U\colon\SCoFib\to\CAT^{\mathbf 2}$ is locally fully faithful.
\end{definition}

The condition in \Cref{eq:cleavage-preservation} is stable under
identities and composition, so this does define a 2-category.  It also
implies that $T$ preserves the factorization in
\Cref{eq:cofib-factorization}:
the vertical factor of $Tf$ is $T\nu_f$.

\subsection{The action associated with a cleavage}

The following elementary construction will be used throughout.

\begin{lemma}\label{lem:action-from-cleavage}
Let $P\colon E\to B$ be a cloven cofibration.  There is a unique
functor
\begin{equation*}
L^{P}\colon P\!\downarrow\!B\longrightarrow E
\end{equation*}
with the following properties:
\begin{align*}
  L^{P}(x,u)&=u_{!}x,\\
  P L^{P}&=\Cod_{P},\\
  L^{P}(f,v)\,\delta^{x}_{u}&=\delta^{y}_{w}f
\end{align*}
for every comma arrow
$(f,v)\colon(x,u)\to(y,w)$, so that $vu=wPf$.
If the cleavage is normal, then
\begin{equation}
L^{P}H^{P}=1_E,
  \label{eq:LH}
\end{equation}
where $H^{P}x=(x,1_{Px})$ and
$H^{P}f=(f,Pf)$.
\end{lemma}

\begin{proof}
The required arrow $L^{P}(f,v)$ is supplied uniquely by the
cocartesianness of $\delta^{x}_{u}$, because
\begin{equation*}
P(\delta^{y}_{w}f)=wPf=vu.
\end{equation*}
Equivalently, it is the unique dashed arrow making
\begin{equation*}
\begin{tikzcd}[column sep=large,row sep=large]
x \arrow[r,"\delta^x_u"] \arrow[d,"f"']
  & u_!x \arrow[d,dashed,"{L^P(f,v)}"] \\
y \arrow[r,"\delta^y_w"']
  & w_!y
\end{tikzcd}
\qquad\text{commute, with}\qquad PL^P(f,v)=v.
\end{equation*}
For an identity comma arrow, both the identity of $u_{!}x$ and the
prescribed image have the same base and agree after
$\delta^{x}_{u}$; hence they agree.  The same argument, now after
precomposing with $\delta^{x}_{u}$, proves preservation of composite
comma arrows.  Thus $L^{P}$ is a functor.  If the cleavage is normal
and $u=1_{Px}$, its identity law says that the displayed formula
reduces to the original arrow $f$.  This proves \Cref{eq:LH}.
\end{proof}

\begin{corollary}\label{cor:split-adjunction}
For every normally cloven cofibration $P$, the functor $L^{P}$ of
\Cref{lem:action-from-cleavage} is left adjoint to $H^{P}$ in the
strict slice 2-category $\CAT/B$.  The counit is the identity, and the
unit at $(x,u)$ is
\begin{equation}
(\delta^{x}_{u},1_b)\colon(x,u)\longrightarrow
  H^{P}(u_{!}x)
  \label{eq:unit-LH}.
\end{equation}
Thus, in particular, there is an adjunction in $\CAT^{\mathbf 2}$
\begin{equation}
(L^{P},1_B)\dashv(H^{P},1_B)
  \label{eq:LP-adjunction}
\end{equation}
over $1_B$.
\end{corollary}

\begin{proof}
For a comma arrow $(f,v)\colon(x,u)\to(y,w)$, the two composites in the
naturality square for \Cref{eq:unit-LH} are
\begin{equation*}
  \bigl(L^{P}(f,v)\delta^{x}_{u},v\bigr)
  \quad\text{and}\quad
  \bigl(\delta^{y}_{w}f,v\bigr),
\end{equation*}
and these are equal by the defining equation for $L^{P}(f,v)$.
Thus the unit is natural by the commutative square
\begin{equation*}
\begin{tikzcd}[column sep=huge,row sep=large]
(x,u) \arrow[r,"{(f,v)}"]
  \arrow[d,"{(\delta^x_u,1_b)}"']
  & (y,w) \arrow[d,"{(\delta^y_w,1_c)}"] \\
H^P(u_!x) \arrow[r,"{H^P L^P(f,v)}"']
  & H^P(w_!y).
\end{tikzcd}
\end{equation*}
\Cref{eq:LH} supplies the identity counit.  The triangle at
$H^{P}x$ is the identity because $\delta^{x}_{1_{Px}}=1_x$.  For the
other triangle, $L^{P}(\delta^{x}_{u},1_b)$ is the unique vertical
arrow whose composite with $\delta^{x}_{u}$ is $\delta^{x}_{u}$; hence
it is $1_{u_{!}x}$.  Thus both triangle equations hold strictly in
$\CAT$, and all displayed transformations lie over $1_B$, so
\Cref{eq:LP-adjunction} is also an adjunction in the arrow
2-category.
\end{proof}

\section{The global comma 2-monad}
\label{sec:monad}

\subsection{The comma 2-functor}

With the notation of \Cref{eq:comma-arrow}, define
\begin{equation*}
\C P:=\Cod_{P}\colon
  P\!\downarrow\!B\longrightarrow B,
  \qquad (x,u)\longmapsto b.
\end{equation*}
This is a split cofibration.  Its designated lift of
$v\colon b\to c$ at $(x,u)$ is
\begin{equation}
(1_x,v)\colon(x,u)\longrightarrow(x,vu).
  \label{eq:free-cleavage}
\end{equation}
Indeed, suppose that $(f,s)\colon(x,u)\to(y,w)$ lies over $s=tv$.
Then $(f,t)\colon(x,vu)\to(y,w)$ is the unique arrow over $t$ whose
composite with $(1_x,v)$ is $(f,s)$, since $tvu=su=wPf$.
Moreover,
\begin{equation*}
  (1_x,1_b)=1_{(x,u)},
  \qquad
  (1_x,w)(1_x,v)=(1_x,wv),
\end{equation*}
so the cleavage is split.
Given a 1-cell $(T,S)\colon P\to Q$ in $\CAT^{\mathbf 2}$, define
\begin{equation*}
\overline T\colon P\!\downarrow\!B
  \longrightarrow Q\!\downarrow\!C
\end{equation*}
by
\begin{align*}
  \overline T(x,u)&=(Tx,Su),\\
  \overline T(f,v)&=(Tf,Sv).
\end{align*}
For a 2-cell
$(\beta,\alpha)\colon(T,S)\Rightarrow(T',S')$, define
\begin{equation}
\overline\beta_{(x,u)}=(\beta_x,\alpha_b)
  \colon(Tx,Su)\longrightarrow(T'x,S'u).
  \label{eq:bar-beta}
\end{equation}
The component in \Cref{eq:bar-beta} satisfies the comma condition by
naturality of $\alpha$ and \Cref{eq:arrow-2-cell}:
\begin{equation*}
\alpha_b Su=S'u\,\alpha_{Px}=S'u\,Q\beta_x.
\end{equation*}
If $(f,v)\colon(x,u)\to(y,w)$ is a comma arrow, then
\begin{equation*}
  (Sv)(Su)=S(vu)=S(wPf)=(Sw)Q(Tf),
\end{equation*}
so $(Tf,Sv)$ is again a comma arrow.  Moreover, naturality of
$\beta$ and $\alpha$ gives
\begin{equation*}
  T'f\,\beta_x=\beta_y\,Tf,
  \qquad
  S'v\,\alpha_b=\alpha_c\,Sv;
\end{equation*}
hence the components in \Cref{eq:bar-beta} form a natural
transformation.  Identities and both compositions are preserved
componentwise.
These assignments define a strict 2-endofunctor
\begin{equation*}
\C\colon\CAT^{\mathbf 2}\longrightarrow\CAT^{\mathbf 2},
\end{equation*}
which fixes the codomain component: $\Cod\C=\Cod$.  In fact,
$(\overline T,S)$ preserves the canonical split cleavages defined in
\Cref{eq:free-cleavage}, so the same assignment gives the free
2-functor
\begin{equation}
\C^{\!*}\colon\CAT^{\mathbf 2}\longrightarrow\SCoFib.
  \label{eq:free-2functor}
\end{equation}
\subsection{Unit, multiplication, and the free adjunction}

Define 2-natural transformations
\begin{equation*}
\eta\colon1\Rightarrow\C,
  \qquad
  \mu\colon\C^2\Rightarrow\C
\end{equation*}
as follows.  At $P\colon E\to B$, their base components are
$1_B$.  Their domain components are
\begin{equation*}
H^{P}\colon E\longrightarrow P\!\downarrow\!B,
  \qquad
  M^{P}\colon \C P\!\downarrow\!B
      \longrightarrow P\!\downarrow\!B,
\end{equation*}
where
\begin{align}
  H^{P}x&=(x,1_{Px}),
  & H^{P}f&=(f,Pf),
  \label{eq:H-definition}\\
  M^{P}((x,u),v)&=(x,vu),
  & M^{P}((f,r),s)&=(f,s).
  \label{eq:M-definition}
\end{align}
In \Cref{eq:M-definition}, an object of $\C^2P$ is written
$((x,u),v)$ with $u\colon Px\to b$ and $v\colon b\to c$; an arrow is
written $((f,r),s)$.
If this arrow has source $((x,u),v)$ and target $((y,w),z)$, its comma
equations are
\begin{equation*}
  ru=wPf,
  \qquad
  sv=zr.
\end{equation*}
Consequently,
\begin{equation*}
  s(vu)=(sv)u=(zr)u=z(ru)=z(wPf)=(zw)Pf,
\end{equation*}
so $(f,s)$ is indeed an arrow from $(x,vu)$ to $(y,zw)$.  Since
$M^P$ retains the total component and the outer base component, it
preserves identities and composites.

\begin{proposition}\label{prop:comma-monad}
The triple $\C=(\C,\eta,\mu)$ is a strict 2-monad on
$\CAT^{\mathbf 2}$.
\end{proposition}

\begin{proof}
We first spell out 2-naturality.  For a square
$(T,S)\colon P\to Q$, direct evaluation gives
\begin{equation*}
  \overline T H^{P}=H^{Q}T,
  \qquad
  \overline T M^{P}=M^{Q}\overline{\overline T},
\end{equation*}
where $\overline{\overline T}$ is the domain component of
$\C^{2}(T,S)$.  On a 2-cell $(\beta,\alpha)$ the corresponding
equalities are
\begin{equation*}
  \overline\beta H^{P}=H^{Q}\beta,
  \qquad
  \overline\beta M^{P}=M^{Q}\overline{\overline\beta}.
\end{equation*}
For example, the first equality has component
$(\beta_x,\alpha_{Px})=(\beta_x,Q\beta_x)$, and the remaining
equalities follow from the same componentwise calculation.  Hence
$\eta$ and $\mu$ are strict 2-natural transformations.

The domain components of the two unit equations are
\begin{equation}
  M^{P}H^{\C P}=1_{P\downarrow B},
  \qquad
  M^{P}\overline{H^{P}}=1_{P\downarrow B}.
  \label{eq:comma-unit-laws}
\end{equation}
At $(x,u)$, the first composite passes through
$((x,u),1_b)$ and the second through $((x,1_{Px}),u)$; both return
$(x,u)$.  On a comma arrow $(f,v)$, each composite returns $(f,v)$.

The domain component of the associativity equation is
\begin{equation}
  M^{P}M^{\C P}=M^{P}\overline{M^{P}}
  \colon\Dom\C^{3}P\longrightarrow P\!\downarrow\!B.
  \label{eq:comma-associativity}
\end{equation}
Equivalently, the following square commutes strictly:
\begin{equation*}
\begin{tikzcd}[column sep=huge,row sep=large]
\Dom\C^3P \arrow[r,"M^{\C P}"]
  \arrow[d,"{\overline{M^P}}"']
& \Dom\C^2P \arrow[d,"M^P"] \\
\Dom\C^2P \arrow[r,"M^P"']
& \Dom\C P .
\end{tikzcd}
\end{equation*}
On an object $(((x,u),v),w)$, its two sides give
$\bigl(x,w(vu)\bigr)$ and $\bigl(x,(wv)u\bigr)$, respectively.
Thus \Cref{eq:comma-associativity} reduces to
\begin{equation*}
w(vu)=(wv)u.
\end{equation*}
On an arrow $(((f,r),s),t)$, both sides give $(f,t)$; hence they retain
the total component and the outermost base component.  This proves
equality on arrows as well.
The base components of all three laws are identities.  Hence all monad
laws hold as equalities of 2-natural transformations.
\end{proof}

Let
\begin{equation*}
U\colon\SCoFib\longrightarrow\CAT^{\mathbf 2}
\end{equation*}
denote the forgetful 2-functor of \Cref{def:scofib}.

\begin{theorem}[Free split cofibration]\label{thm:free-adjunction}
The 2-functor $\C^{\!*}$ in \Cref{eq:free-2functor} is left
2-adjoint to $U$.  Its unit is $\eta$.  At a split cofibration $P$,
the counit is
\begin{equation}
\varepsilon_P=(L^{P},1_B)\colon\C P\longrightarrow P.
  \label{eq:counit}
\end{equation}
The induced 2-monad is the comma 2-monad of
\Cref{prop:comma-monad}.
\end{theorem}

\begin{proof}
Let $P\colon E\to B$ be arbitrary and let $Q\colon F\to C$ be a
split cofibration.  Precomposition with $(H^{P},1_B)$ gives a functor
\begin{equation}
\SCoFib(\C P,Q)\longrightarrow
  \CAT^{\mathbf 2}(P,UQ).
  \label{eq:hom-map}
\end{equation}
We construct a strict inverse.  Given
$(T,S)\colon P\to Q$, put
\begin{equation}
\widetilde T=L^{Q}\overline T,
  \qquad
  \widetilde T(x,u)=(Su)_{!}Tx.
  \label{eq:T-tilde}
\end{equation}
By \Cref{lem:action-from-cleavage}, the image of a comma arrow
$(f,v)\colon(x,u)\to(y,w)$ is the unique arrow over $Sv$ satisfying
\begin{equation}
\widetilde T(f,v)\,\delta^{Tx}_{Su}
  =\delta^{Ty}_{Sw}Tf.
  \label{eq:T-tilde-arrow}
\end{equation}
This gives a square $(\widetilde T,S)\colon\C P\to Q$.
The splitting laws for $Q$ show from \Cref{eq:T-tilde-arrow} that it
preserves the canonical cleavage \Cref{eq:free-cleavage}; explicitly,
\begin{equation*}
\widetilde T(1_x,v)\delta^{Tx}_{Su}
  =\delta^{Tx}_{S(vu)}
  =\delta^{(Su)_{!}Tx}_{Sv}\delta^{Tx}_{Su},
\end{equation*}
and hence, by cocartesian uniqueness,
\begin{equation*}
\widetilde T(1_x,v)=\delta^{(Su)_{!}Tx}_{Sv}.
\end{equation*}
Moreover $\widetilde T H^{P}=T$ by the identity clauses of the
splitting.
Thus the extension fills the triangle
\begin{equation*}
\begin{tikzcd}[column sep=huge,row sep=large]
P \arrow[r,"{(H^P,1_B)}"]
  \arrow[dr,"{(T,S)}"']
& \C P \arrow[d,"{(\widetilde T,S)}"] \\
& Q .
\end{tikzcd}
\end{equation*}

Conversely, a cleavage-preserving extension of $(T,S)$ must send
$(x,u)$ to $(Su)_{!}Tx$ and $(1_x,u)$ to
$\delta^{Tx}_{Su}$.  Every comma arrow factors as
\begin{equation*}
(f,v)=(f,1_c)\,(1_x,v)
\end{equation*}
with the evident intermediate object, and its vertical factor is
forced by cocartesian uniqueness.  More explicitly, one has
\begin{equation*}
  (f,1_c)(1_x,vu)=(1_y,w)H^{P}f,
\end{equation*}
so any extension $K$ restricting to $(T,S)$ must satisfy
\begin{equation*}
  K(f,1_c)\,\delta^{Tx}_{S(vu)}=\delta^{Ty}_{Sw}Tf.
\end{equation*}
The first arrow is therefore uniquely determined, and so is
$K(f,v)=K(f,1_c)K(1_x,v)$.  Hence the extension must be the functor in
\Cref{eq:T-tilde}.

For a 2-cell $(\beta,\alpha)$ between two squares $P\to Q$, its
extension has component at $(x,u)$
\begin{equation}
L^{Q}(\beta_x,\alpha_b).
  \label{eq:extended-2cell}
\end{equation}
The comma equation was checked after \Cref{eq:bar-beta}; naturality
follows from functoriality of $L^{Q}$.  Restriction and extension are
inverse also on 2-cells.  More explicitly,
\Cref{eq:extended-2cell} evaluated at $H^{P}x$ gives
$L^{Q}H^{Q}(\beta_x)=\beta_x$, since
$Q\beta_x=\alpha_{Px}$.  In the other direction, naturality with the
canonical lift $(1_x,u)$ forces the component at $(x,u)$ to be
\Cref{eq:extended-2cell}.  Thus \Cref{eq:hom-map} is an isomorphism of
categories.  Its formulas commute strictly with precomposition in
$P$.  For postcomposition, if $(R,V)\colon Q\to Q'$ preserves the
designated cleavages, then both $R L^{Q}\overline T$ and
$L^{Q'}\overline{RT}$ are the unique cleavage-preserving extension of
$(RT,VS)$, and hence they are equal as functors.  The same argument
applies to 2-cells: the arrows
$R L^Q(\beta_x,\alpha_b)$ and
$L^{Q'}(R\beta_x,V\alpha_b)$ lie over $V\alpha_b$, and their
composites with the chosen lift of $VS u$ both equal the image of
$\beta_x$.  Cocartesian uniqueness makes them equal.  Thus the
hom-category isomorphism is strict
2-natural in both variables, proving the 2-adjunction.

At $P=UQ$ and $(T,S)=1_Q$, the formula in \Cref{eq:T-tilde} is $L^{Q}$,
giving the stated counit.  Applying $U$ to the free construction and
to the counit gives precisely the formulas for $\C$ and $M^{P}$:
\begin{equation*}
  L^{\C P}((x,u),v)=(x,vu),
  \qquad
  L^{\C P}((f,r),s)=(f,s).
\end{equation*}
Thus $L^{\C P}=M^P$ strictly.
\end{proof}

\subsection{Lax idempotence}

We use the convention that a 2-monad $\mathbb T$ is \emph{lax
idempotent}, or of Kock--Z\"oberlein type, if
$\mu_X\dashv\eta_{\mathbb T X}$ with identity counit, naturally in
$X$; see \citet{Zoberlein1976,Kock1995}.
Equivalently, $\mu\dashv\eta_{\mathbb T}$ in the 2-category of strict
2-endofunctors, 2-natural transformations, and modifications.  Thus
the unit of this adjunction must itself be a modification, not merely a
family of pointwise natural transformations.

\begin{proposition}\label{prop:kz}
The comma 2-monad $\C$ is lax idempotent.  More precisely, for every
$P\colon E\to B$ there is an adjunction in $\CAT^{\mathbf 2}$
\begin{equation}
(M^{P},1_B)\dashv(H^{\C P},1_B)
  \label{eq:kz-adjunction}
\end{equation}
with identity counit.
\end{proposition}

\begin{proof}
The monad law gives
$M^{P}H^{\C P}=1_{P\downarrow B}$, so take the counit to be the
identity.  Define the unit
\begin{equation*}
\sigma^{P}\colon1_{\C^2P}\Longrightarrow
  H^{\C P}M^{P}
\end{equation*}
at an object $((x,u),v)$ by the arrow
\begin{equation}
\sigma^{P}_{((x,u),v)}=((1_x,v),1_c)
  \colon((x,u),v)\longrightarrow((x,vu),1_c),
  \label{eq:kz-unit}
\end{equation}
where $u\colon Px\to b$ and $v\colon b\to c$.  This is an arrow of
$\C^2P$: its inner base component is $v$, and its outer base component
is $1_c$.  The comma equations are therefore
$vu=(vu)P1_x$ and $1_cv=1_cv$.

Consider an arrow
\begin{equation*}
  ((f,r),s)\colon((x,u),v)\longrightarrow((y,w),z)
\end{equation*}
of $\C^2P$.  Its two comma equations are
\begin{equation*}
  ru=wPf,
  \qquad
  sv=zr.
\end{equation*}
In the naturality square for $\sigma^{P}$, both composites have inner
total component $f$ and outer base component $s$.  Their inner base
components are $sv$ and $zr$, respectively, and these are equal by the
second comma equation.  Thus $\sigma^{P}$ is natural and lies over
$1_B$.

The family $\sigma^{P}$ is natural also in the object $P$ of the arrow
2-category: for every square $(T,S)\colon P\to Q$ one has
\begin{equation}
  \C^2(T,S)\,\sigma^P
  =\sigma^Q\,\C^2(T,S).
  \label{eq:kz-modification-naturality}
\end{equation}
If $Z=((x,u),v)$ is an object of $\Dom\C^2P$, then
\Cref{eq:kz-modification-naturality} is the commutative square
\begin{equation*}
\begin{tikzcd}[column sep=huge,row sep=large]
Z \arrow[r,"{\overline{\overline T}}"]
  \arrow[d,"{\sigma^P_Z}"']
& \overline{\overline T}Z
  \arrow[d,"{\sigma^Q_{\overline{\overline T}Z}}"] \\
H^{\C P}M^PZ
  \arrow[r,"{\overline{\overline T}}"']
& H^{\C Q}M^Q\overline{\overline T}Z .
\end{tikzcd}
\end{equation*}
Indeed, applying $\C^2(T,S)$ to the component in
\Cref{eq:kz-unit} gives
\begin{equation*}
  ((1_{Tx},Sv),1_{Sc}),
\end{equation*}
which is the component of $\sigma^{Q}$ at
$((Tx,Su),Sv)$.  The same componentwise equality is compatible with a
2-cell $(\beta,\alpha)$ between squares.  Thus the units form the
modification
\begin{equation*}
  \sigma\colon 1_{\C^{2}}\Longrightarrow\eta_{\C}\mu
\end{equation*}
required by the adjunction $\mu\dashv\eta_{\C}$.

At an object in the image of $H^{\C P}$, the component defined in
\Cref{eq:kz-unit} is an identity because $v$ is an identity.  Also,
$M^{P}\sigma^{P}$ is the identity, since $M^{P}$ retains the total
component $1_x$ and the outer component $1_c$.  These are the two
triangle identities.  Hence \Cref{eq:kz-adjunction} is an adjunction
with identity counit.
\end{proof}

\section{Strict algebras are split cofibrations}
\label{sec:monadicity}

\subsection{The strict Eilenberg--Moore 2-category}

We first make the 2-dimensional target of the comparison explicit.
An object of $\Alg(\C)$ is a functor $P\colon E\to B$ together with a
1-cell
\begin{equation*}
a=(A,R)\colon\C P\longrightarrow P
\end{equation*}
in $\CAT^{\mathbf 2}$ such that
\begin{equation}
a\eta_P=1_P,
  \qquad
  a\mu_P=a\C a.
  \label{eq:abstract-algebra-laws}
\end{equation}
Since the base component of $\eta_P$ is $1_B$, the first equality
forces $R=1_B$.  We may therefore write every action as
\begin{equation*}
(A,1_B)\colon\C P\longrightarrow P,
  \qquad
  PA=\Cod_P.
\end{equation*}
On domain components the algebra laws are
\begin{equation}
AH^{P}=1_E,
  \qquad
  AM^{P}=A\overline A,
  \label{eq:algebra-laws}
\end{equation}
where $\overline A$ is the domain component of
$\C(A,1_B)\colon\C^2P\to\C P$.

A strict algebra morphism
\begin{equation*}
(T,S)\colon(P,A)\longrightarrow(Q,D)
\end{equation*}
is a 1-cell in $\CAT^{\mathbf 2}$ satisfying
\begin{equation}
TA=D\overline T.
  \label{eq:strict-algebra-morphism}
\end{equation}
An algebra 2-cell from $(T,S)$ to $(T',S')$ is a 2-cell
$(\beta,\alpha)$ in $\CAT^{\mathbf 2}$ satisfying
\begin{equation}
\beta A=D\overline\beta.
  \label{eq:algebra-2cell}
\end{equation}
The component of the right-hand side at $(x,u)$ is
$D(\beta_x,\alpha_b)$, with $\overline\beta$ as in
\Cref{eq:bar-beta}.

\subsection{From a strict action to a split cleavage}

Let $(P,A)$ be a strict $\C$-algebra.  For
$u\colon Px\to b$, define
\begin{equation}
u_{!}x:=A(x,u),
  \qquad
  \delta^{x}_{u}:=A(1_x,u)
  \colon x\longrightarrow u_{!}x.
  \label{eq:delta-from-action}
\end{equation}
Here $(1_x,u)\colon H^{P}x\to(x,u)$ is the canonical arrow of
$P\!\downarrow\!B$.  Since $PA=\Cod_P$, one has
$P\delta^{x}_{u}=u$.

The following identity is the key consequence of the multiplication
law that is not visible from the unit law alone.

\begin{lemma}[Absorption]\label{lem:absorption}
For every $(x,u)\in P\!\downarrow\!B$, the comma arrow
\begin{equation*}
(\delta^{x}_{u},1_b)\colon(x,u)\longrightarrow H^{P}(u_{!}x)
\end{equation*}
satisfies
\begin{equation}
A(\delta^{x}_{u},1_b)=1_{u_{!}x}.
  \label{eq:absorption}
\end{equation}
\end{lemma}

\begin{proof}
Write $a=Px$.  In $\C^2P$, consider the arrow
\begin{equation}
\varpi_{x,u}=((1_x,u),1_b)
  \colon ((x,1_a),u)\longrightarrow((x,u),1_b).
  \label{eq:absorption-test-arrow}
\end{equation}
Under $M^{P}$, both its source and target become $(x,u)$ and the
arrow becomes the identity $(1_x,1_b)$.  Hence
\begin{equation*}
AM^{P}(\varpi_{x,u})=1_{u_{!}x}.
\end{equation*}
On the other hand, $\overline A$ sends the source of
$\varpi_{x,u}$ to $(x,u)$, because $AH^{P}=1_E$, and it sends the
target to $H^{P}(u_{!}x)$.  Its value on $\varpi_{x,u}$ is
$(\delta^{x}_{u},1_b)$.  Applying the multiplication law
$AM^{P}=A\overline A$ to \Cref{eq:absorption-test-arrow} gives
\Cref{eq:absorption}.
\begin{samepage}
The two images of $\varpi_{x,u}$ are displayed by
\begin{equation*}
\begin{tikzcd}[column sep=huge,row sep=large]
((x,1_a),u) \arrow[r,"\varpi_{x,u}"] \arrow[d,"M^P"']
  & ((x,u),1_b) \arrow[d,"M^P"] \\
(x,u) \arrow[r,"1_{(x,u)}"'] & (x,u)
\end{tikzcd}
\end{equation*}
and
\begin{equation*}
\begin{tikzcd}[column sep=huge,row sep=large]
((x,1_a),u) \arrow[r,"\varpi_{x,u}"] \arrow[d,"{\overline A}"']
  & ((x,u),1_b) \arrow[d,"{\overline A}"] \\
(x,u) \arrow[r,"{(\delta^x_u,1_b)}"'] & H^P(u_!x).
\end{tikzcd}
\end{equation*}
\end{samepage}
\end{proof}

\begin{corollary}\label{cor:algebra-adjunction}
For every strict $\C$-algebra $(P,A)$, there is an adjunction in
$\CAT^{\mathbf 2}$
\begin{equation}
  (A,1_B)\dashv(H^{P},1_B)
  \label{eq:algebra-adjunction}
\end{equation}
with identity counit.  Its unit at $(x,u)$ is
\begin{equation}
  \tau_{(x,u)}=(A(1_x,u),1_b)
  \colon(x,u)\longrightarrow H^{P}A(x,u).
  \label{eq:algebra-adjunction-unit}
\end{equation}
\end{corollary}

\begin{proof}
Write $\delta^{x}_{u}=A(1_x,u)$.  For a comma arrow
$(f,v)\colon(x,u)\to(y,w)$, functoriality of $A$ gives
\begin{equation}
  A(f,v)\delta^{x}_{u}=\delta^{y}_{w}f.
  \label{eq:algebra-unit-naturality}
\end{equation}
The total components of the two sides of the naturality square for
$\tau$ are the two sides of \Cref{eq:algebra-unit-naturality}, while
both base components are $v$.  Thus $\tau$ is a natural transformation
over $1_B$; explicitly,
\begin{equation*}
\begin{tikzcd}[column sep=huge,row sep=large]
(x,u) \arrow[r,"{(f,v)}"]
  \arrow[d,"{(\delta^x_u,1_b)}"']
& (y,w) \arrow[d,"{(\delta^y_w,1_c)}"] \\
H^PA(x,u) \arrow[r,"{H^PA(f,v)}"']
& H^PA(y,w)
\end{tikzcd}
\end{equation*}
commutes.  The algebra unit law $AH^{P}=1_E$ gives the identity
counit.  At $H^{P}x$, the unit component is the identity by the same
algebra law, proving the triangle equation at the right adjoint.  The
other triangle equation is
\begin{equation*}
  A(\delta^{x}_{u},1_b)=1_{A(x,u)};
\end{equation*}
it is the absorption identity of \Cref{lem:absorption}.  Hence
\Cref{eq:algebra-adjunction} is an adjunction in the arrow
2-category.
\end{proof}

For the splitting corresponding to $A$, this is the adjunction
\Cref{eq:LP-adjunction}.  We shall not use adjointness to prove
monadicity; the direct argument below exposes the algebra equations
responsible for cocartesianness and splitting.

\begin{proposition}\label{prop:algebra-gives-splitting}
The arrows defined in \Cref{eq:delta-from-action} form a split cleavage for
$P$.  Thus the underlying functor of every strict $\C$-algebra is a
split cofibration.
\end{proposition}

\begin{proof}
We prove cocartesianness first.  Suppose that $g\colon x\to z$ and
$w\colon b\to Pz$ satisfy $Pg=wu$.  Then
$(g,w)\colon(x,u)\to H^{P}z$ is a comma arrow.  Set
\begin{equation*}
\widehat g:=A(g,w)\colon u_{!}x\longrightarrow z.
\end{equation*}
It is the dashed arrow in
\begin{equation*}
\begin{tikzcd}[column sep=large,row sep=large]
x \arrow[r,"\delta^x_u"] \arrow[dr,"g"']
& u_!x \arrow[d,dashed,"\widehat g"] \\
& z,
\end{tikzcd}
\qquad P\widehat g=w,
\end{equation*}
and functoriality together with the unit law gives
\begin{equation}
\widehat g\,\delta^{x}_{u}
  =A(g,w)A(1_x,u)
  =A(g,Pg)
  =AH^{P}(g)
  =g.
  \label{eq:action-factorization}
\end{equation}
If $h\colon u_{!}x\to z$ is any other arrow over $w$ with
$h\delta^{x}_{u}=g$, then \Cref{eq:absorption} yields
\begin{align*}
  h
  &=AH^{P}(h)\\
  &=AH^{P}(h)\,A(\delta^{x}_{u},1_b)\\
  &=A(g,w)=\widehat g.
\end{align*}
Indeed, in the comma category the composite of
$(\delta^{x}_{u},1_b)$ with $H^{P}(h)=(h,w)$ is $(g,w)$.
This proves the universal property stated in \Cref{eq:cocartesian-up}.

The unit action law gives
\begin{equation}
(1_{Px})_{!}x=AH^{P}x=x,
  \qquad
  \delta^{x}_{1_{Px}}=AH^{P}(1_x)=1_x.
  \label{eq:action-unit-split}
\end{equation}
For composable $u\colon Px\to b$ and $v\colon b\to c$, applying
$AM^{P}=A\overline A$ to the object $((x,u),v)$ gives
\begin{equation}
(vu)_{!}x=v_{!}(u_{!}x).
  \label{eq:action-object-split}
\end{equation}
It remains to check the cleavage arrows.  In $\C^2P$, take
\begin{equation}
\omega_{x,u,v}=((1_x,u),vu)
  \colon((x,1_{Px}),1_{Px})\longrightarrow((x,u),v).
  \label{eq:multiplication-test-arrow}
\end{equation}
The left side of the multiplication law sends this arrow to
$\delta^{x}_{vu}$.  For the right side, $\overline A$ sends it to
the comma arrow
\begin{equation*}
(\delta^{x}_{u},vu)
  \colon H^{P}x\longrightarrow(u_{!}x,v).
\end{equation*}
This arrow factors in $P\!\downarrow\!B$ as
\begin{equation*}
\begin{tikzcd}[column sep=huge]
H^{P}x
  \arrow[r,"{(1_x,u)}"]
& (x,u)
  \arrow[r,"{(\delta^{x}_{u},v)}"]
& (u_{!}x,v),
\end{tikzcd}
\end{equation*}
and the second arrow factors as
\begin{equation*}
(\delta^{x}_{u},v)
  =(1_{u_{!}x},v)(\delta^{x}_{u},1_b).
\end{equation*}
Applying $A$, then using \Cref{eq:delta-from-action} and
\Cref{eq:absorption}, shows that the right side of the algebra law
applied to \Cref{eq:multiplication-test-arrow} is
$\delta^{u_{!}x}_{v}\delta^{x}_{u}$.  Hence
\begin{equation}
\delta^{x}_{vu}=\delta^{u_{!}x}_{v}\delta^{x}_{u}.
  \label{eq:action-delta-split}
\end{equation}
For a vertical $f\colon x\to y$ in the fibre over $a$, the pushforward
induced by this cleavage is
$u_{!}f=A(f,1_b)$, viewing $(f,1_b)\colon(x,u)\to(y,u)$ as a comma
arrow.  The defining equality in \Cref{eq:pushforward-arrow} follows by
functoriality of $A$, as do
\begin{equation*}
  u_{!}(1_x)=1_{u_{!}x},
  \qquad
  u_{!}(gf)=(u_{!}g)(u_{!}f).
\end{equation*}
The unit law gives $(1_a)_{!}f=f$.  Applying the multiplication law
to an arrow between $((x,u),v)$ and $((y,u),v)$ with components
$((f,1_b),1_c)$ gives
\begin{equation*}
(vu)_{!}f=v_{!}(u_{!}f).
\end{equation*}
Together with
\Cref{eq:action-unit-split,eq:action-object-split,eq:action-delta-split},
these are all the split-cleavage axioms.
\end{proof}

\subsection{From a split cleavage to a strict action}

We now prove that \Cref{lem:action-from-cleavage} supplies the inverse
construction.

\begin{proposition}\label{prop:splitting-gives-algebra}
If $P\colon E\to B$ is a split cofibration, then
\begin{equation*}
(L^{P},1_B)\colon\C P\longrightarrow P
\end{equation*}
is a strict $\C$-algebra action.
\end{proposition}

\begin{proof}
The equality in \Cref{eq:LH} is the unit law.  We prove
\begin{equation}
L^{P}M^{P}=L^{P}\overline{L^{P}}.
  \label{eq:LP-multiplication}
\end{equation}
At an object $((x,u),v)$, the left side is $(vu)_{!}x$ and the right
side is $v_{!}(u_{!}x)$, so equality follows from
\Cref{eq:split-functors}.

Consider an arbitrary arrow in $\C^2P$,
\begin{equation}
\Xi=((f,r),s)
  \colon((x,u),v)\longrightarrow((y,w),z).
  \label{eq:Xi}
\end{equation}
Thus
\begin{equation}
ru=wPf,
  \qquad
  sv=zr.
  \label{eq:Xi-relations}
\end{equation}
The arrow $L^{P}M^{P}(\Xi)$ lies over $s$ and is characterized by
\begin{equation}
L^{P}M^{P}(\Xi)\,\delta^{x}_{vu}
  =\delta^{y}_{zw}f.
  \label{eq:LM-characterization}
\end{equation}
The other side of \Cref{eq:LP-multiplication} is the arrow over $s$
characterized by
\begin{equation}
L^{P}\overline{L^{P}}(\Xi)\,
  \delta^{u_{!}x}_{v}
  =\delta^{w_{!}y}_{z}L^{P}(f,r).
  \label{eq:LL-characterization}
\end{equation}
The two defining lift squares form the diagram
\begin{equation}
\begin{tikzcd}[column sep=large,row sep=large]
x \arrow[r,"\delta^x_u"] \arrow[d,"f"']
& u_!x \arrow[r,"\delta^{u_!x}_v"]
  \arrow[d,"{L^P(f,r)}"]
& v_!u_!x \arrow[d,"{L^P\overline{L^P}(\Xi)}"] \\
y \arrow[r,"\delta^y_w"']
& w_!y \arrow[r,"\delta^{w_!y}_z"']
& z_!w_!y .
\end{tikzcd}
\label{eq:LP-multiplication-diagram}
\end{equation}
The upper and lower composites are $\delta^x_{vu}$ and
$\delta^y_{zw}$ by the split cleavage law.
Precompose \Cref{eq:LL-characterization} with $\delta^{x}_{u}$.
The split equation in \Cref{eq:split-delta} and the defining property of
$L^{P}(f,r)$ give
\begin{equation*}
L^{P}\overline{L^{P}}(\Xi)\,\delta^{x}_{vu}
  =\delta^{w_{!}y}_{z}\delta^{y}_{w}f
  =\delta^{y}_{zw}f.
\end{equation*}
Consequently the two arrows both satisfy
\Cref{eq:LM-characterization}.  Cocartesian uniqueness makes them
equal.  This proves \Cref{eq:LP-multiplication} on arrows and hence the
strict algebra law.
\end{proof}

\begin{proposition}\label{prop:object-inverses}
The constructions of \Cref{prop:algebra-gives-splitting} and
\Cref{prop:splitting-gives-algebra} are mutually inverse strictly.
\end{proposition}

\begin{proof}
Starting from a split cleavage, \Cref{eq:delta-from-action} applied to
$L^{P}$ recovers
\begin{equation*}
L^{P}(1_x,u)\delta^x_{1_{Px}}=\delta^{x}_{u}.
\end{equation*}
By normality $\delta^x_{1_{Px}}=1_x$, and therefore
$L^{P}(1_x,u)=\delta^{x}_{u}$.  Thus the chosen cleavage is unchanged.

Conversely, start from an action $A$, form its cleavage by
\Cref{eq:delta-from-action}, and let $L^{P}$ be the functor constructed
from that cleavage.  For a comma arrow
$(f,v)\colon(x,u)\to(y,w)$, functoriality of $A$ gives
\begin{equation}
A(f,v)\delta^{x}_{u}
  =A(f,v)A(1_x,u)
  =A(f,wPf)
  =\delta^{y}_{w}f.
  \label{eq:A-characterization}
\end{equation}
Both $A(f,v)$ and $L^{P}(f,v)$ lie over $v$.  The cocartesian
uniqueness proved in \Cref{prop:algebra-gives-splitting} therefore
gives $A(f,v)=L^{P}(f,v)$.  They also agree on objects, so
$A=L^{P}$ as functors.
\end{proof}

\subsection{Morphisms and 2-cells}

\begin{proposition}\label{prop:one-cells}
Let $(P,L^{P})$ and $(Q,L^{Q})$ be the algebras associated with split
cofibrations.  A square $(T,S)\colon P\to Q$ is a strict
$\C$-algebra morphism if and only if it preserves the designated
cocartesian liftings strictly.
\end{proposition}

\begin{proof}
Suppose first that the strict algebra equation
\begin{equation}
TL^{P}=L^{Q}\overline T
  \label{eq:TL-equality}
\end{equation}
holds.  Evaluate it at the canonical comma arrow
$(1_x,u)\colon H^{P}x\to(x,u)$.  The left side is
$T\delta^{x}_{u}$ and the right side is
$\delta^{Tx}_{Su}$, so $(T,S)$ preserves the cleavage.

Conversely, assume \Cref{eq:cleavage-preservation}.  On objects it
gives
\begin{equation*}
T(u_{!}x)=(Su)_{!}Tx,
\end{equation*}
which is \Cref{eq:TL-equality} on objects.  For a comma arrow
$(f,v)\colon(x,u)\to(y,w)$, the arrow $TL^{P}(f,v)$ lies over $Sv$
and appears in
\begin{equation*}
\begin{tikzcd}[column sep=huge,row sep=large]
Tx \arrow[r,"\delta^{Tx}_{Su}"] \arrow[d,"Tf"']
& (Su)_!Tx \arrow[d,"{TL^P(f,v)}"] \\
Ty \arrow[r,"\delta^{Ty}_{Sw}"']
& (Sw)_!Ty .
\end{tikzcd}
\end{equation*}
After precomposition with $\delta^{Tx}_{Su}=T\delta^{x}_{u}$, it
therefore gives
\begin{equation*}
T(\delta^{y}_{w}f)=\delta^{Ty}_{Sw}Tf.
\end{equation*}
By definition, $L^{Q}(Tf,Sv)$ is the unique arrow over $Sv$ with
this property.  Hence $TL^{P}(f,v)=L^{Q}\overline T(f,v)$, proving
\Cref{eq:TL-equality}.
\end{proof}

The corresponding assertion for 2-cells is where the local fullness
in \Cref{def:scofib} meets the Eilenberg--Moore condition.

\begin{proposition}\label{prop:two-cells}
Let $(T,S),(T',S')\colon P\to Q$ preserve the designated split
cleavages.  Every 2-cell
\begin{equation*}
(\beta,\alpha)\colon(T,S)\Rightarrow(T',S')
\end{equation*}
in $\CAT^{\mathbf 2}$ automatically satisfies the algebra 2-cell
equation
\begin{equation}
\beta L^{P}=L^{Q}\overline\beta.
  \label{eq:auto-algebra-2cell}
\end{equation}
\end{proposition}

\begin{proof}
Fix $(x,u)$ with $u\colon Px\to b$.  The two sides of
\Cref{eq:auto-algebra-2cell} at $(x,u)$ are arrows
\begin{equation*}
(Su)_{!}Tx=T(u_{!}x)
  \longrightarrow
  T'(u_{!}x)=(S'u)_{!}T'x
\end{equation*}
over $\alpha_b$.  Precomposing the left side with the designated
lift of $Su$ gives, by naturality of $\beta$ and strict preservation
of the cleavages,
\begin{align}
  \beta_{u_{!}x}\,\delta^{Tx}_{Su}
  &=\beta_{u_{!}x}T\delta^{x}_{u}\notag\\
  &=T'\delta^{x}_{u}\,\beta_x\notag\\
  &=\delta^{T'x}_{S'u}\,\beta_x.
  \label{eq:2cell-precompose}
\end{align}
The component of the right side is
$L^{Q}(\beta_x,\alpha_b)$, and its defining property is precisely
the last expression in \Cref{eq:2cell-precompose}.  Since
$\delta^{Tx}_{Su}$ is $Q$-cocartesian, uniqueness proves
\Cref{eq:auto-algebra-2cell}.
The cancellation is displayed by
\begin{equation*}
\begin{tikzcd}[column sep=huge,row sep=large]
Tx \arrow[r,"\delta^{Tx}_{Su}"] \arrow[d,"\beta_x"']
& T(u_!x) \arrow[d,"\beta_{u_!x}"] \\
T'x \arrow[r,"\delta^{T'x}_{S'u}"']
& T'(u_!x),
\end{tikzcd}
\end{equation*}
whose image under $Q$ is the naturality square for $\alpha$ at $u$.
\end{proof}

\begin{theorem}[Strict 2-monadicity]\label{thm:main}
The comparison 2-functor
\begin{equation}
\Gamma\colon\SCoFib\longrightarrow\Alg(\C),
  \qquad
  P\longmapsto(P,L^{P}),
  \label{eq:comparison}
\end{equation}
is an isomorphism of 2-categories over $\CAT^{\mathbf 2}$.
Consequently the forgetful 2-functor
\begin{equation*}
U\colon\SCoFib\longrightarrow\CAT^{\mathbf 2}
\end{equation*}
is strictly 2-monadic.
\end{theorem}

\begin{proof}
By \Cref{prop:object-inverses}, the object function of $\Gamma$ has a
strict inverse: it sends $(P,A)$ to the split cleavage defined by
\Cref{eq:delta-from-action}, and the recovered action is $A$.
By \Cref{prop:one-cells}, the strict algebra morphisms are precisely
the 1-cells of $\SCoFib$, with no change in their underlying squares.
By \Cref{prop:two-cells}, the algebra 2-cells are precisely all
2-cells of $\CAT^{\mathbf 2}$ between those squares, again with no
change in the underlying pairs of natural transformations.  These
identifications leave the underlying squares and 2-cells unchanged and
preserve identities and compositions.  They therefore define a strict
inverse 2-functor.  Both functors commute with the forgetful
2-functors to $\CAT^{\mathbf 2}$.
\end{proof}

\begin{corollary}\label{cor:ordinary-monadicity}
After forgetting 2-cells, the category of split cofibrations and
strictly cleavage-preserving squares is the Eilenberg--Moore category
of the underlying ordinary monad of $\C$ on the underlying category of
$\CAT^{\mathbf 2}$.
\end{corollary}

\begin{proof}
Take underlying categories in the isomorphism of
\Cref{thm:main}.
\end{proof}

\section{Normal pseudoalgebras and cloven cofibrations}
\label{sec:pseudoalgebras}

Strict $\C$-algebras encode split transport.  We now identify the
coherent structure obtained when the multiplication equation is replaced
by an invertible 2-cell.  Normality is important: it removes a
base-changing component that is present in an arbitrary pseudoalgebra on
the global arrow 2-category.

\subsection{The concrete pseudoalgebra data}

Let $P\colon E\to B$.  A normal pseudo-$\C$-algebra on $P$ has an
action whose base component is forced to be the identity,
\begin{equation*}
  (A,1_B)\colon\C P\longrightarrow P,
  \qquad PA=\Cod_P,
\end{equation*}
and whose normal unit equation is
\begin{equation}
  AH^P=1_E.
  \label{eq:normal-pseudo-unit}
\end{equation}
The compositor is initially a 2-cell in the arrow 2-category and could
therefore have a base component
$\zeta\colon1_B\Rightarrow1_B$.  Either normal unit axiom, evaluated on
base components, forces $\zeta=1_{1_B}$.  Its total component is thus an
invertible natural transformation over $1_B$,
\begin{equation}
  \chi\colon A\overline A\Longrightarrow AM^P.
  \label{eq:C-pseudo-compositor}
\end{equation}
Writing
\begin{equation*}
  u_!x=A(x,u),
\end{equation*}
the component of \Cref{eq:C-pseudo-compositor} at
$((x,u),v)$ is a vertical isomorphism
\begin{equation}
  \chi^x_{v,u}\colon v_!(u_!x)\xrightarrow{\ \cong\ }(vu)_!x.
  \label{eq:transport-compositor}
\end{equation}
The two pseudoalgebra unit axioms say
\begin{equation}
  \chi^x_{1_b,u}=1_{u_!x},
  \qquad
  \chi^x_{u,1_{Px}}=1_{u_!x}.
  \label{eq:transport-unit-coherence}
\end{equation}
For composable arrows
\begin{equation*}
  Px\xrightarrow{u}b\xrightarrow{v}c\xrightarrow{w}d,
\end{equation*}
the associativity axiom is the commutativity of the transport pentagon
\begin{equation}
\begin{tikzcd}[column sep=huge,row sep=large]
w_!v_!u_!x
  \arrow[r,"{w_!\chi^x_{v,u}}"]
  \arrow[d,"{\chi^{u_!x}_{w,v}}"']
& w_!(vu)_!x
  \arrow[d,"{\chi^x_{w,vu}}"] \\
(wv)_!u_!x
  \arrow[r,"{\chi^x_{wv,u}}"']
& (wvu)_!x .
\end{tikzcd}
\label{eq:transport-pentagon}
\end{equation}
These equations are the usual normal pseudoalgebra triangles and
pentagon specialized to the formulas for $H^P$ and $M^P$.

An arbitrary global pseudoalgebra need not have this form.  Its action
is a square
\begin{equation*}
  (A,R)\colon\C P\longrightarrow P
\end{equation*}
and its unit and multiplication are invertible 2-cells in the arrow
2-category
\begin{equation}
  (\iota,\rho)\colon(1_E,1_B)\Longrightarrow(AH^P,R),
  \qquad
  (\chi,\zeta)\colon(A\overline A,R^2)
    \Longrightarrow(AM^P,R).
  \label{eq:global-pseudoalgebra-cells}
\end{equation}
Their components satisfy the arrow-2-category equations
\begin{equation}
  P\iota=\rho P,
  \qquad
  P\chi=\zeta\,\Cod_{\C P}.
  \label{eq:global-pseudoalgebra-cell-conditions}
\end{equation}
In particular, the base component of the pseudo-unit is an isomorphism
\begin{equation}
  \rho\colon1_B\Longrightarrow R.
  \label{eq:base-twist}
\end{equation}
The base component of its multiplication cell is an isomorphism
\begin{equation*}
  \zeta\colon R^2\Longrightarrow R.
\end{equation*}
The base components of the two unit axioms and the associativity axiom
are
\begin{equation}
  \zeta(\rho R)=1_R=\zeta(R\rho),
  \qquad
  \zeta(\zeta R)=\zeta(R\zeta).
  \label{eq:global-base-coherence}
\end{equation}
Thus $(R,\rho,\zeta)$ is a pseudoalgebra for the identity 2-monad at
$B$.  Since $\rho$ is invertible, either unit equation determines
$\zeta$ from $\rho$; the base action is coherently, but not necessarily
strictly, trivial.
If $u\colon Px\to b$, the arrow obtained from the action and unit is
displayed by the commutative square
\begin{equation}
\begin{tikzcd}[column sep=huge,row sep=large]
x \arrow[r,"{A(1_x,u)\,\iota_x}"] \arrow[d,mapsto]
  & A(x,u) \arrow[d,mapsto] \\
Px \arrow[r,"{Ru\,\rho_{Px}=\rho_bu}"']
  & Rb .
\end{tikzcd}
\label{eq:global-base-twisted-lift}
\end{equation}
Thus its target lies over $Rb$, and the arrow lies over
$Ru\,\rho_{Px}=\rho_bu\colon Px\to Rb$.
It is therefore not a chosen lift over $u$ in the ordinary sense.
Normality makes the unit cell an identity and forces $R=1_B$; this is
why \Cref{eq:normal-pseudo-unit} is available.  The same phenomenon is
analyzed, in the fibrational variance, by
\citet[Remark~2.7(b)]{EmmeneggerEtAl2024}.  Their unrestricted global
pseudoalgebras carry a uniform base isomorphism and give Street
fibrations in the sense of \citet{Street1974}, although the converse
need not hold.  Dually, the
underlying functor here is a Street opfibration (not necessarily an
ordinary Grothendieck cofibration), and the converse need not hold
because the required base isomorphism must be uniform.  We return to
the fixed-base alternative after the global theorem.

\subsection{Canonical coherence of a normal cleavage}

We first record the pseudoalgebra coherence already present in a
cleavage.

\begin{lemma}[Transport coherence]\label{lem:transport-coherence}
Let $P\colon E\to B$ be a normally cloven cofibration.  For composable
$u\colon Px\to b$ and $v\colon b\to c$, there is a unique vertical
isomorphism
\begin{equation}
  \chi^x_{v,u}\colon v_!(u_!x)\xrightarrow{\ \cong\ }(vu)_!x
  \label{eq:cleavage-compositor}
\end{equation}
satisfying
\begin{equation}
  \chi^x_{v,u}\,
  \delta^{u_!x}_{v}\,
  \delta^x_u
  =\delta^x_{vu}.
  \label{eq:cleavage-compositor-characterization}
\end{equation}
The isomorphisms are natural and satisfy
\Cref{eq:transport-unit-coherence,eq:transport-pentagon}.
Consequently, the functor $L^P$ of
\Cref{lem:action-from-cleavage}, together with these cells, is a normal
pseudo-$\C$-algebra.
\end{lemma}

\begin{proof}
The composite of two cocartesian arrows is cocartesian.  Indeed, a map
out of its domain is first factored uniquely through $\delta^x_u$ and
then uniquely through $\delta^{u_!x}_v$; the two uniqueness statements
also give uniqueness for the composite.  Hence the two arrows
\begin{equation*}
  \delta^{u_!x}_{v}\delta^x_u,
  \qquad
  \delta^x_{vu}
\end{equation*}
are cocartesian over the same base arrow $vu$ and have the same domain.
The universal property of the first gives the arrow in
\Cref{eq:cleavage-compositor}; the universal property of the second
gives an arrow in the reverse direction.  Their composites are
identities by cocartesian uniqueness, so $\chi^x_{v,u}$ is invertible.
The defining relation is the commutative transport triangle
\begin{equation}
\begin{tikzcd}[column sep=large,row sep=large]
x \arrow[r,"\delta^x_u"]
  \arrow[drr,bend right=20,"\delta^x_{vu}"']
& u_!x \arrow[r,"\delta^{u_!x}_v"]
& v_!u_!x \arrow[d,"{\chi^x_{v,u}}"] \\
&& (vu)_!x .
\end{tikzcd}
\label{eq:cleavage-compositor-triangle}
\end{equation}

For an arrow
$\Xi=((f,r),s)\colon((x,u),v)\to((y,w),z)$ of $\C^2P$, naturality is
the commutativity of
\begin{equation}
\begin{tikzcd}[column sep=huge,row sep=large]
v_!u_!x \arrow[r,"\chi^x_{v,u}"]
  \arrow[d,"{L^P(L^P(f,r),s)}"']
& (vu)_!x \arrow[d,"{L^P(f,s)}"] \\
z_!w_!y \arrow[r,"\chi^y_{z,w}"']
& (zw)_!y .
\end{tikzcd}
\label{eq:cleavage-compositor-naturality}
\end{equation}
After precomposition with
$\delta^{u_!x}_v\delta^x_u$, both paths become
$\delta^y_{zw}f$ by the defining equations of $L^P$ and
\Cref{eq:cleavage-compositor-characterization}.  Cocartesian uniqueness
therefore proves \Cref{eq:cleavage-compositor-naturality}.  Normality gives the
two equations in \Cref{eq:transport-unit-coherence}.  In the pentagon
\Cref{eq:transport-pentagon}, precompose both paths with
\begin{equation*}
  \delta^{v_!u_!x}_w\,
  \delta^{u_!x}_v\,
  \delta^x_u.
\end{equation*}
Repeated use of
\Cref{eq:cleavage-compositor-characterization} turns either path into
$\delta^x_{wvu}$.  Since the displayed composite of lifts is
cocartesian, the two paths are equal.  Thus the normal unit and
associativity axioms hold.  The action unit is
$L^PH^P=1_E$ by normality and \Cref{lem:action-from-cleavage}, so
$(L^P,\chi)$ is a normal pseudoalgebra.
\end{proof}

\begin{definition}\label{def:nclcofib}
The 2-category $\NClCoFib$ has normally cloven cofibrations as objects.
A 1-cell $(T,S)\colon P\to Q$ is a commutative square for which $T$
preserves cocartesian arrows.  Equivalently, it is enough that every
designated arrow $\delta^x_u$ be sent to a $Q$-cocartesian arrow.
All 2-cells of $\CAT^{\mathbf 2}$ between such squares are admitted.
The evident forgetful 2-functor to $\CAT^{\mathbf 2}$ is locally fully
faithful.
\end{definition}

For the equivalence in \Cref{def:nclcofib}, observe that any
cocartesian arrow over $u$ is a designated lift followed by a vertical
isomorphism.  Thus preservation of designated lifts as cocartesian
arrows implies preservation of all cocartesian arrows.

\subsection{The normal pseudoalgebra comparison}

We next reconstruct a normal cleavage from a pseudoaction.  The proof
also gives the adjunction-to-the-unit form of lax idempotence in this
particular case; compare \citet[Sections~2.27--2.30]{Street1980} and
\citet{Kock1995}.

\begin{lemma}[Pseudo-absorption]\label{lem:pseudo-absorption}
Let $(P,A,\chi)$ be a normal pseudo-$\C$-algebra.  Put
\begin{equation}
  u_!x=A(x,u),
  \qquad
  \delta^x_u=A(1_x,u).
  \label{eq:pseudo-delta-from-action}
\end{equation}
Then
\begin{equation}
  A(\delta^x_u,1_b)=1_{u_!x}.
  \label{eq:pseudo-absorption}
\end{equation}
Consequently $(A,1_B)\dashv(H^P,1_B)$ in
$\CAT^{\mathbf 2}$, with identity counit and unit
\begin{equation*}
  (\delta^x_u,1_b)\colon(x,u)\longrightarrow H^P(u_!x).
\end{equation*}
In particular, every $\delta^x_u$ is $P$-cocartesian.
\end{lemma}

\begin{proof}
Consider the arrow $\varpi_{x,u}$ of
\Cref{eq:absorption-test-arrow}.  The two functors in
\Cref{eq:C-pseudo-compositor} send it as indicated in
\begin{equation*}
\begin{tikzcd}[column sep=huge,row sep=large]
A\overline A((x,1_{Px}),u)=u_!x
  \arrow[r,"{A(\delta^x_u,1_b)}"]
  \arrow[d,"\chi"']
& A\overline A((x,u),1_b)=u_!x
  \arrow[d,"\chi"] \\
AM^P((x,1_{Px}),u)=u_!x
  \arrow[r,"1_{u_!x}"']
& AM^P((x,u),1_b)=u_!x .
\end{tikzcd}
\end{equation*}
The vertical components are identities by the two normal unit axioms
\Cref{eq:transport-unit-coherence}; the bottom arrow is an identity
because $M^P(\varpi_{x,u})=1_{(x,u)}$.  Naturality of $\chi$ therefore
gives \Cref{eq:pseudo-absorption}.

Functoriality of $A$ gives the naturality of the proposed adjunction
unit, exactly as in \Cref{eq:algebra-unit-naturality}.  The normal unit
law in \Cref{eq:normal-pseudo-unit} gives the identity counit.  The two
triangle identities are normality at $H^Px$ and
\Cref{eq:pseudo-absorption}, respectively.

Suppose $g\colon x\to z$ and $w\colon b\to Pz$ satisfy $Pg=wu$.  The
comma arrow $(g,w)\colon(x,u)\to H^Pz$ gives
$\widehat g=A(g,w)$, and
\begin{equation*}
\begin{tikzcd}[column sep=large,row sep=large]
x \arrow[r,"\delta^x_u"] \arrow[dr,"g"']
& u_!x \arrow[d,dashed,"\widehat g"] \\
& z
\end{tikzcd}
\qquad P\widehat g=w.
\end{equation*}
Functoriality and \Cref{eq:normal-pseudo-unit} give
$\widehat g\delta^x_u=g$.  If $h$ is another such arrow, then
\begin{equation*}
  h=AH^P(h)A(\delta^x_u,1_b)=A(g,w)=\widehat g
\end{equation*}
by \Cref{eq:pseudo-absorption}.  This is the universal property in
\Cref{eq:cocartesian-up}.
\end{proof}

The component of the pseudoalgebra cell of a normal pseudoaction is the
canonical transport comparison.  More precisely, applying naturality of
$\chi$ to the arrow of $\C^2P$ used in the proof of
\Cref{prop:algebra-gives-splitting}, and then using
\Cref{eq:pseudo-absorption}, gives
\begin{equation}
  \chi^x_{v,u}\,
  \delta^{u_!x}_{v}\,
  \delta^x_u=\delta^x_{vu}.
  \label{eq:pseudo-compositor-characterization}
\end{equation}
Thus it agrees with the unique comparison in
\Cref{lem:transport-coherence}.

\begin{theorem}[Normal pseudoalgebras]\label{thm:normal-pseudo}
There is an isomorphism of 2-categories over $\CAT^{\mathbf 2}$
\begin{equation}
  \Gamma_{\mathrm n}\colon
  \NClCoFib\xrightarrow{\ \cong\ }\PsAlg_{\mathrm n}(\C).
  \label{eq:normal-pseudo-comparison}
\end{equation}
It sends a normally cloven cofibration to the action and compositor of
\Cref{lem:action-from-cleavage,lem:transport-coherence}.  Under this
isomorphism:
\begin{enumerate}[label=\textup{(\roman*)},leftmargin=2.2em]
\item pseudomorphisms are the squares whose total functors
      preserve cocartesian arrows;
\item strict pseudoalgebra morphisms are the squares that preserve the
      designated lifts strictly;
\item every arrow-2-category 2-cell between pseudomorphisms is an
      algebra transformation.
\end{enumerate}
\end{theorem}

\begin{proof}
On objects, \Cref{lem:transport-coherence} constructs a normal
pseudoalgebra from a normal cleavage, while
\Cref{lem:pseudo-absorption} constructs a normal cleavage from a normal
pseudoalgebra.  Starting with a cleavage, the recovered designated
arrows are
\begin{equation*}
  L^P(1_x,u)=\delta^x_u.
\end{equation*}
Conversely, let $(A,\chi)$ be a normal pseudoaction.  For every comma
arrow $(f,v)\colon(x,u)\to(y,w)$, functoriality gives
\begin{equation*}
  A(f,v)\delta^x_u=\delta^y_wf.
\end{equation*}
Since $\delta^x_u$ is cocartesian by
\Cref{lem:pseudo-absorption}, this equation characterizes
$A(f,v)=L^P(f,v)$.  The actions therefore agree as functors, and
\Cref{eq:pseudo-compositor-characterization} identifies their
structural cells.  The object assignments are inverse.

Let $(T,S)\colon P\to Q$ be a commutative square, and write $A$ and
$D$ for the two actions.  Any pseudomorphism structure on this square
is initially a 2-cell in $\CAT^{\mathbf 2}$ with a possible base
component $\sigma\colon S\Rightarrow S$.  Normal unit coherence forces
$\sigma=1_S$ and forces its total component at $H^Px$ to be
$1_{Tx}$.  We may therefore write the comparison as the natural
transformation over $1_S$
\begin{equation}
  \theta\colon D\overline T\Longrightarrow TA.
  \label{eq:C-pseudomorphism-comparison}
\end{equation}
Naturality of $\theta$ at the comma arrow
$(1_x,u)\colon H^Px\to(x,u)$ then gives
\begin{equation}
\begin{tikzcd}[column sep=huge,row sep=large]
Tx \arrow[r,"\delta^{Tx}_{Su}"]
  \arrow[dr,"T\delta^x_u"']
& (Su)_!Tx \arrow[d,"{\theta_{x,u}}"] \\
& T(u_!x).
\end{tikzcd}
\label{eq:pseudomorphism-lift-triangle}
\end{equation}
If $\theta$ is invertible, $T\delta^x_u$ is an isomorphism after a
$Q$-cocartesian arrow and is therefore cocartesian.  Hence $T$
preserves cocartesian arrows.

Conversely, suppose that $T$ preserves cocartesian arrows.  Both
$\delta^{Tx}_{Su}$ and $T\delta^x_u$ are cocartesian over $Su$ with
domain $Tx$.  There is a unique vertical isomorphism
\begin{equation*}
  \theta_{x,u}\colon(Su)_!Tx\xrightarrow{\ \cong\ }T(u_!x)
\end{equation*}
making \Cref{eq:pseudomorphism-lift-triangle} commute.  Naturality
is the commutativity, for every comma arrow
$(f,v)\colon(x,u)\to(y,w)$, of
\begin{equation}
\begin{tikzcd}[column sep=huge,row sep=large]
(Su)_!Tx \arrow[r,"{\theta_{x,u}}"]
  \arrow[d,"{L^Q(Tf,Sv)}"']
  & T(u_!x) \arrow[d,"{T L^P(f,v)}"] \\
(Sw)_!Ty \arrow[r,"{\theta_{y,w}}"']
  & T(w_!y).
\end{tikzcd}
\label{eq:pseudomorphism-naturality-square}
\end{equation}
Indeed, after precomposition with $\delta^{Tx}_{Su}$, both paths are
$T\delta^y_w\,Tf$; cocartesian uniqueness gives the square.

\begin{samepage}
\noindent
The multiplication axiom is expressed by the commutative diagram
\begin{equation}
\begin{tikzcd}[column sep=large,row sep=large]
(Sv)_!(Su)_!Tx
  \arrow[r,"\chi^Q"]
  \arrow[d,"{(Sv)_!\theta_{x,u}}"']
& (S(vu))_!Tx
  \arrow[d,"{\theta_{x,vu}}"] \\
(Sv)_!T(u_!x)
  \arrow[d,"{\theta_{u_!x,v}}"']
& T((vu)_!x) \\
T(v_!u_!x)
  \arrow[r,"T\chi^P"']
& T((vu)_!x) \arrow[u,equal]
\end{tikzcd}
\label{eq:pseudomorphism-composition}
\end{equation}
where the base functor $S$ is suppressed from the subscripts.
\end{samepage}
After precomposition with the chosen composite lift, both paths become
$T\delta^x_{vu}$; cocartesian uniqueness proves the equality.  The
specialization of \Cref{eq:pseudomorphism-lift-triangle} to
$u=1_{Px}$, together with normality, gives
$\theta_{x,1_{Px}}=1_{Tx}$, which is precisely the pseudomorphism unit
axiom.  Thus $\theta$ is the unique pseudomorphism structure on
$(T,S)$.  It is an identity if and only if the designated lifts are
preserved strictly, proving
Parts~\textup{(i)} and \textup{(ii)}.

Finally, let
$(\beta,\alpha)\colon(T,S)\Rightarrow(T',S')$ be a 2-cell of
$\CAT^{\mathbf 2}$.  The algebra-transformation equation at $(x,u)$ is
\begin{equation}
  \beta_{u_!x}\theta_{x,u}
  =\theta'_{x,u}D(\beta_x,\alpha_b).
  \label{eq:pseudo-algebra-transformation}
\end{equation}
After precomposition with $\delta^{Tx}_{Su}$, both sides reduce, by
\Cref{eq:pseudomorphism-lift-triangle} and naturality of $\beta$, to
$T'\delta^x_u\,\beta_x$.  Since the chosen lift is cocartesian,
\Cref{eq:pseudo-algebra-transformation} follows.  The object, 1-cell,
and 2-cell assignments preserve identities and compositions because
all comparison cells are uniquely characterized by their lift
triangles.  They are therefore inverse 2-functors over
$\CAT^{\mathbf 2}$.
\end{proof}

\begin{remark}\label{rem:noncanonical-pseudomorphisms}
The morphism criterion in \Cref{thm:normal-pseudo} also has a
non-canonical-isomorphism formulation.  A lax morphism has a cell of the
form displayed in \Cref{eq:C-pseudomorphism-comparison}, without the requirement that
it be invertible.  Since $\C$ is lax idempotent,
\citet[Corollary~3.5]{LucatelliNunes2019} implies that its canonical
structure cell is invertible whenever there exists an invertible
2-cell of the prescribed underlying shape, even if that cell is not the
canonical algebra comparison.  By \Cref{thm:normal-pseudo},
invertibility of the canonical comparison is equivalent to preservation of
cocartesian arrows.  Thus the non-canonical-isomorphism criterion detects
when a lax morphism of the transport pseudoalgebras is a
pseudomorphism.
\end{remark}

\begin{theorem}[Fixed-base pseudoalgebras]
\label{thm:fixed-base-pseudo}
Let $\C_B$ be the restriction of $\C$ to $\CAT/B$, and let
$\ClCoFib(B)$ be the 2-category of cloven cofibrations over $B$,
cocartesian-arrow-preserving functors over $B$, and natural
transformations over $1_B$.  There is an isomorphism of 2-categories
over $\CAT/B$
\begin{equation}
  \ClCoFib(B)\xrightarrow{\ \cong\ }\PsAlg(\C_B).
  \label{eq:fixed-base-pseudo-comparison}
\end{equation}
It restricts to the isomorphism between normally cloven cofibrations and
normal pseudoalgebras.  Under the isomorphism in
\Cref{eq:fixed-base-pseudo-comparison},
strict pseudoalgebra morphisms are the functors preserving the
chosen lifts strictly.
\end{theorem}

\begin{proof}
Start with a cleavage, not assumed normal, and put $A=L^P$.  Define the
pseudoalgebra unit by
\begin{equation}
  \iota_x:=\delta^x_{1_{Px}}\colon x\longrightarrow(1_{Px})_!x
  =AH^Px.
  \label{eq:fixed-base-pseudo-unit}
\end{equation}
It is invertible because a cocartesian arrow over an identity is an
isomorphism.  Naturality is the defining equation for $L^P$ applied to
$H^Pf$.  For composable $u$ and $v$, let $\chi^x_{v,u}$ be the unique
vertical isomorphism characterized by
\begin{equation*}
  \chi^x_{v,u}\delta^{u_!x}_v\delta^x_u=\delta^x_{vu}.
\end{equation*}
The proof of \Cref{lem:transport-coherence} does not use normality for
the existence, naturality, or pentagon of these isomorphisms.  The two
non-normal unit triangles are
\begin{align}
  \chi^x_{1_b,u}\,\iota_{u_!x}&=1_{u_!x},
  \label{eq:fixed-base-left-unit}\\
  \chi^x_{u,1_{Px}}\,L^P(\iota_x,1_b)&=1_{u_!x}.
  \label{eq:fixed-base-right-unit}
\end{align}
The first is the defining comparison equation with the identity
transport at $u_!x$.  For the second, functoriality of $L^P$ gives
\begin{equation}
  L^P(\iota_x,1_b)\delta^x_u
  =\delta^{(1_{Px})_!x}_u\delta^x_{1_{Px}}.
  \label{eq:fixed-base-right-unit-naturality}
\end{equation}
After precomposition with $\delta^x_u$, the left-hand side of
\Cref{eq:fixed-base-right-unit} is therefore $\delta^x_u$, by the
defining equation for $\chi^x_{u,1_{Px}}$.  Since $\delta^x_u$ is
cocartesian, cancellation proves the second unit triangle.  Thus
$(L^P,\iota,\chi)$ is a pseudo-$\C_B$-algebra.

Conversely, let $(A,\iota,\chi)$ be a pseudo-$\C_B$-algebra.  Define
the chosen lift of $u\colon Px\to b$ by
\begin{equation}
  \delta^x_u:=A(1_x,u)\,\iota_x
  \colon x\longrightarrow A(x,u).
  \label{eq:fixed-base-lift-from-pseudoaction}
\end{equation}
Naturality of $\iota$ and functoriality of $A$ give
\begin{equation}
  A(f,v)\delta^x_u=\delta^y_wf
  \label{eq:fixed-base-unit-naturality}
\end{equation}
for every comma arrow $(f,v)\colon(x,u)\to(y,w)$.  Applying naturality
of $\chi$ to $\varpi_{x,u}$ from
\Cref{eq:absorption-test-arrow} gives the commutative square
\begin{equation}
\begin{tikzcd}[column sep=huge,row sep=large]
A(A(x,1_{Px}),u)
  \arrow[r,"{A(A(1_x,u),1_b)}"]
  \arrow[d,"{\chi^x_{u,1_{Px}}}"']
& A(A(x,u),1_b)
  \arrow[d,"{\chi^x_{1_b,u}}"] \\
A(x,u) \arrow[r,equal]
& A(x,u).
\end{tikzcd}
\label{eq:fixed-base-absorption-naturality}
\end{equation}
Equivalently,
\begin{equation*}
  \chi^x_{1_b,u}A(A(1_x,u),1_b)=\chi^x_{u,1_{Px}}.
\end{equation*}
The two non-normal unit triangles say
\begin{equation*}
  \chi^x_{1_b,u}\iota_{A(x,u)}=1_{A(x,u)},
  \qquad
  \chi^x_{u,1_{Px}}A(\iota_x,1_b)=1_{A(x,u)}.
\end{equation*}
Precomposing the equality represented by
\Cref{eq:fixed-base-absorption-naturality} with
$A(\iota_x,1_b)$, and using
$\delta^x_u=A(1_x,u)\iota_x$, shows that both
$A(\delta^x_u,1_b)$ and $\iota_{A(x,u)}$ become the identity after
postcomposition with the invertible cell $\chi^x_{1_b,u}$.
Cancelling that cell gives the generalized absorption identity
\begin{equation}
  A(\delta^x_u,1_b)=\iota_{A(x,u)}.
  \label{eq:fixed-base-pseudo-absorption}
\end{equation}
\Cref{eq:fixed-base-unit-naturality,eq:fixed-base-pseudo-absorption}
give an adjunction in $\CAT/B$
\begin{equation}
  A\dashv H^P
  \label{eq:fixed-base-pseudo-adjunction}
\end{equation}
whose counit is $\iota^{-1}\colon AH^P\Rightarrow1_E$ and whose unit
at $(x,u)$ is $(\delta^x_u,1_b)$.  The first equation gives naturality
of the unit.  The triangle at $A(x,u)$ is
$\iota^{-1}_{A(x,u)}A(\delta^x_u,1_b)=1$, by the second equation; the
triangle at $H^Px$ is
$H^P(\iota_x^{-1})(\delta^x_{1_{Px}},1_{Px})=1$.

To see the cocartesian universal property directly, consider an arrow
\begin{equation*}
  (g,w)\colon(x,u)\longrightarrow H^Pz.
\end{equation*}
It corresponds under
\Cref{eq:fixed-base-pseudo-adjunction} to a unique arrow
$h\colon A(x,u)\to z$ over $w$, and the unit equation says precisely
$h\delta^x_u=g$.  Hence \Cref{eq:fixed-base-lift-from-pseudoaction} is
cocartesian.

The constructions are strict inverses.  Starting from a cleavage, the
lift reconstructed from its action is
\begin{equation*}
  L^P(1_x,u)\delta^x_{1_{Px}}=\delta^x_u.
\end{equation*}
Starting from a pseudoaction, naturality of $\iota$ and functoriality of
$A$ show that, for every comma arrow $(f,v)\colon(x,u)\to(y,w)$,
\begin{equation*}
  A(f,v)\delta^x_u=\delta^y_wf.
\end{equation*}
Cocartesian uniqueness therefore identifies $A$ with the functor $L^P$
reconstructed from the cleavage.  \Cref{eq:fixed-base-pseudo-unit}
recovers $\iota$.  It remains to recover the compositor.  Put
$a=Px$, and let $u\colon a\to b$ and $v\colon b\to c$.  Naturality of
the original compositor at the arrow $\omega_{x,u,v}$ of
\Cref{eq:multiplication-test-arrow} is the commutativity of
\begin{equation}
\begin{tikzcd}[column sep=huge,row sep=large]
A(A(x,1_a),1_a)
  \arrow[r,"{A(A(1_x,u),vu)}"]
  \arrow[d,"{\chi^x_{1_a,1_a}}"']
& A(A(x,u),v)
  \arrow[d,"{\chi^x_{v,u}}"] \\
A(x,1_a)
  \arrow[r,"{A(1_x,vu)}"']
& A(x,vu).
\end{tikzcd}
\label{eq:fixed-base-compositor-naturality}
\end{equation}
Precompose this square with
$A(\iota_x,1_a)\iota_x$.  The lower path becomes
$A(1_x,vu)\iota_x=\delta^x_{vu}$ by the right unit triangle.
For the upper path, naturality of $\iota$ at $\delta^x_u$ gives
\begin{equation*}
  \iota_{A(x,u)}\delta^x_u
  =A(\delta^x_u,u)\iota_x,
\end{equation*}
and hence the upper path, before its final $\chi$-component, is
$\delta^{u_!x}_v\delta^x_u$.  Therefore
\Cref{eq:fixed-base-compositor-naturality} gives
\begin{equation}
  \chi^x_{v,u}\delta^{u_!x}_v\delta^x_u=\delta^x_{vu}.
  \label{eq:fixed-base-recover-compositor}
\end{equation}
The composite $\delta^{u_!x}_v\delta^x_u$ is cocartesian, so
precomposition with it is cancellable among arrows over the prescribed
base.  Hence \Cref{eq:fixed-base-recover-compositor} uniquely determines
$\chi$ and proves that the compositor is also recovered.
It remains to identify 1-cells and 2-cells.  Let $T\colon E\to F$ be a
functor over $B$ between two cloven cofibrations.  If $T$ preserves
cocartesian arrows, the comparison component is the unique vertical
isomorphism
\begin{equation}
  \theta_{x,u}\colon u_!Tx\xrightarrow{\ \cong\ }T(u_!x)
  \label{eq:fixed-base-pseudomorphism-comparison}
\end{equation}
such that
\begin{equation}
  \theta_{x,u}\delta^{Tx}_u=T\delta^x_u.
  \label{eq:fixed-base-pseudomorphism-triangle}
\end{equation}
At $u=1_{Px}$, this reads
$\theta_{x,1_{Px}}\iota_{Tx}=T\iota_x$, which is the pseudomorphism
unit axiom.  For composable $u$ and $v$, precomposing the two sides of
the multiplication axiom with
$\delta^{u_!Tx}_v\delta^{Tx}_u$ reduces both to
$T\delta^x_{vu}$; cocartesian cancellation proves the axiom.  The
same triangle shows conversely that a pseudomorphism sends each chosen
lift, hence every cocartesian arrow, to a cocartesian arrow.

For an algebra transformation $\beta\colon T\Rightarrow T'$, the two
sides of its compatibility equation become
$T'\delta^x_u\,\beta_x$ after precomposition with $\delta^{Tx}_u$.
Cocartesian cancellation therefore proves the equation.  Finally,
$\theta$ is an identity if and only if
$T\delta^x_u=\delta^{Tx}_u$.  Thus pseudomorphisms are the
cocartesian-arrow-preserving functors, algebra transformations are the natural
transformations over $1_B$, and strict morphisms preserve the chosen
lifts strictly.  This proves
\Cref{eq:fixed-base-pseudo-comparison}.  The result is the
cofibrational dual of \citet{Street1974}; compare
\citet[Theorem~2.6]{EmmeneggerEtAl2024}.
\end{proof}

\section{From Grothendieck cofibrations to factorizations via 2-monads}
\label{sec:squaring}

Every arrow in the total category of a cloven cofibration factors as a
chosen cocartesian arrow followed by a vertical arrow.  Taken alone,
this is an elementary consequence of the cocartesian universal
property.  The purpose of this section is stronger: we show that the
passage from Grothendieck transport to these factorizations is induced
by a morphism of 2-monads.  At the strict level it is functorial in
split cofibrations, cleavage-preserving squares, and 2-cells; at the
level of normal pseudoalgebras it is functorial in normally cloven cofibrations,
cocartesian-arrow-preserving squares, and 2-cells.  No algebraic data
are added beyond the comma-monad actions identified in
\Cref{thm:main,thm:normal-pseudo}.

\subsection{The squaring 2-monad}

The walking arrow $\mathbf 2$ is a comonoid in the cartesian monoidal
category $\Cat$: its counit is the unique functor
$\mathbf 2\to\mathbf 1$, and its comultiplication is the diagonal
$\mathbf 2\to\mathbf 2\times\mathbf 2$.  Exponentiation therefore
defines the \emph{squaring 2-monad}
\begin{equation*}
\Sq\colon\CAT\longrightarrow\CAT,
  \qquad
  \Sq E=E^{\mathbf 2}.
\end{equation*}
Its unit
\begin{equation*}
I_E\colon E\longrightarrow E^{\mathbf 2}
\end{equation*}
sends an object to its identity arrow.  Its multiplication
\begin{equation*}
m_E\colon(E^{\mathbf 2})^{\mathbf 2}\longrightarrow E^{\mathbf 2}
\end{equation*}
sends a commutative square to its diagonal:
\begin{equation}
\begin{tikzcd}[column sep=large,row sep=large]
x \arrow[r,"r"] \arrow[d,"f"']
  & x' \arrow[d,"g"] \\
y \arrow[r,"s"']
  & y'
\end{tikzcd}
\qquad\longmapsto\qquad
sf=gr\colon x\longrightarrow y'.
\label{eq:square-multiplication}
\end{equation}
On a morphism of commutative squares, represented by a commutative
cube, $m_E$ retains the map between the two upper-left vertices and the
map between the two lower-right vertices.  Both associativity composites
select the long diagonal of a commutative cube of squares, and the unit
composites insert and then remove identity sides.  Hence the comonoid
equations for $\mathbf 2$ give the 2-monad equations as strict
equalities on objects, arrows, functors, and natural transformations.
A strict $\Sq$-algebra is thus a functor
$F\colon E^{\mathbf 2}\to E$ satisfying
\begin{equation}
FI_E=1_E,
  \qquad
  Fm_E=F(F^{\mathbf 2}).
  \label{eq:square-algebra-laws}
\end{equation}
For an arrow $f\colon x\to y$, functoriality supplies a factorization
\begin{equation}
f=m_f e_f,
  \qquad
  e_f:=F(1_x,f)\colon x\to Ff,
  \qquad
  m_f:=F(f,1_y)\colon Ff\to y.
  \label{eq:sq-factorization}
\end{equation}
Indeed, in $E^{\mathbf 2}$ one has
\begin{equation*}
  (f,f)=(f,1_y)(1_x,f)=I_E(f).
\end{equation*}
Applying $F$ and using $FI_E=1_E$ gives $f=m_fe_f$.
Strict $\Sq$-algebras are the strict factorization algebras of
\citet{KorostenskiTholen1993}.  Their correspondence with strict
factorization systems is also obtained by
\citet{RosebrughWood2002Distributive}; their coherence was studied by
\citet{RosebrughWood2002}, while the lax variant was developed by
\citet{RosickyTholen2002}.

Normal pseudo-$\Sq$-algebras have an equally precise interpretation.
Let $\OFSch$ denote the 2-category whose objects are orthogonal
factorization systems equipped with factorizations
$f=m_fe_f$ chosen so that the factorization of an identity is the pair
of identities, whose 1-cells preserve both factorization classes, and
whose 2-cells are natural transformations.  By
\citet[Theorem~2.4]{KorostenskiTholen1993}, there is an isomorphism of
2-categories
\begin{equation}
  \PsAlg_{\mathrm n}(\Sq)\cong\OFSch.
  \label{eq:sq-pseudo-OFS}
\end{equation}
The action sends $f$ to its chosen middle object; its invertible
multiplication cell is the unique comparison between the iterated and
the chosen one-step factorization.  Under
the isomorphism in \Cref{eq:sq-pseudo-OFS}, pseudomorphisms are the functors
preserving both orthogonal classes.  Strict morphisms preserve the
chosen factorizations strictly.

A strict $\Sq$-algebra should not be confused with a natural, or
algebraic, weak factorization system.  The latter tradition encodes
functorial weak factorizations by pointed lax indexed endofunctors
\citep{JanelidzeTholen1999}, or by compatible comonad and monad
structures over the arrow category \citep{GrandisTholen2006}.  Within
this latter tradition, \citet{ClementinoLopezFranco2016} introduced lax
orthogonal algebraic weak factorization systems on 2-categories.  Here a
single strict $\Sq$-action chooses the middle object and the two factors
of every arrow.  The ordinary orthogonal classes arise only after
closing the strict system under isomorphisms.

A strict morphism $T\colon(E,F)\to(E',G)$ satisfies
$TF=GT^{\mathbf 2}$, and a 2-cell $\beta\colon T\Rightarrow T'$ between
strict morphisms is an algebra 2-cell when
$\beta F=G\beta^{\mathbf 2}$.

\subsection{Changing 2-monads}

The mechanism used below is the strict 2-dimensional version of
restriction of scalars.

\begin{definition}\label{def:colax-monad-morphism}
Let $\mathbb T$ be a strict 2-monad on $\mathcal K$ and let
$\mathbb R$ be a strict 2-monad on $\mathcal L$.  A \emph{colax
morphism of 2-monads} from $\mathbb T$ to $\mathbb R$ consists of a
strict 2-functor $V\colon\mathcal K\to\mathcal L$ and a 2-natural
transformation
\begin{equation*}
  \lambda\colon\mathbb R V\Longrightarrow V\mathbb T
\end{equation*}
such that, for every $X\in\mathcal K$,
\begin{align}
  \lambda_X\eta^{\mathbb R}_{VX}
    &=V\eta^{\mathbb T}_X,
    \label{eq:general-monad-unit}\\
  \lambda_X\mu^{\mathbb R}_{VX}
    &=V\mu^{\mathbb T}_X\,
      \lambda_{\mathbb T X}\,
      \mathbb R\lambda_X.
    \label{eq:general-monad-multiplication}
\end{align}
\end{definition}

The words ``lax'' and ``colax'' are interchanged in parts of the
literature.  The orientation of $\lambda$ and the two displayed
equations are the convention relevant here.

\begin{lemma}[Restriction of scalars]\label{lem:restriction-scalars}
Every colax morphism $(V,\lambda)$ as in
\Cref{def:colax-monad-morphism} induces a strict 2-functor
\begin{equation*}
  \lambda^{*}\colon\Alg(\mathbb T)\longrightarrow\Alg(\mathbb R)
\end{equation*}
over $V$.  It sends a strict action
$a\colon\mathbb T X\to X$ to
\begin{equation}
  Va\,\lambda_X\colon\mathbb R(VX)\longrightarrow VX,
  \label{eq:restricted-action}
\end{equation}
and sends algebra morphisms and algebra 2-cells by applying $V$.
The square
\begin{equation*}
\begin{tikzcd}[column sep=huge,row sep=large]
\Alg(\mathbb T) \arrow[r,"\lambda^{*}"]
  \arrow[d,"U_{\mathbb T}"']
  & \Alg(\mathbb R) \arrow[d,"U_{\mathbb R}"] \\
\mathcal K \arrow[r,"V"'] & \mathcal L
\end{tikzcd}
\end{equation*}
commutes strictly.
\end{lemma}

\begin{proof}
For the unit law, \Cref{eq:general-monad-unit} and the unit law for
$a$ give
\begin{equation*}
  Va\,\lambda_X\eta^{\mathbb R}_{VX}
  =Va\,V\eta^{\mathbb T}_X
  =1_{VX}.
\end{equation*}
For multiplication, use
\Cref{eq:general-monad-multiplication}, the multiplication law for
$a$, and 2-naturality of $\lambda$ at
$a\colon\mathbb T X\to X$:
\begin{align*}
Va\,\lambda_X\mu^{\mathbb R}_{VX}
  &=Va\,V\mu^{\mathbb T}_X
      \lambda_{\mathbb T X}\mathbb R\lambda_X\\
  &=Va\,V(\mathbb T a)
      \lambda_{\mathbb T X}\mathbb R\lambda_X\\
  &=Va\,\lambda_X\mathbb R(Va)\mathbb R\lambda_X\\
  &=(Va\,\lambda_X)\mathbb R(Va\,\lambda_X).
\end{align*}
If $f\colon(X,a)\to(Y,b)$ is a strict $\mathbb T$-morphism, then
2-naturality gives
\begin{equation*}
  Vf\,Va\,\lambda_X
  =Vb\,V\mathbb T f\,\lambda_X
  =Vb\,\lambda_Y\mathbb R(Vf),
\end{equation*}
so $Vf$ is a strict $\mathbb R$-morphism.  The same 2-naturality
equation applied to an algebra 2-cell
$\alpha\colon f\Rightarrow g$ gives the required law explicitly:
\begin{align*}
V\alpha\,(Va\,\lambda_X)
  &=V(\alpha a)\lambda_X
   =V(b\,\mathbb T\alpha)\lambda_X\\
  &=Vb\,V(\mathbb T\alpha)\lambda_X
   =Vb\,\lambda_Y\,\mathbb R(V\alpha).
\end{align*}
Since $V$ is strict, identities and both compositions are preserved on
the nose.  The asserted equality of forgetful 2-functors follows from
the object, 1-cell, and 2-cell formulas.
\end{proof}

\begin{proposition}[Pseudoalgebraic restriction of scalars]
\label{prop:pseudo-restriction-scalars}
Every colax morphism $(V,\lambda)$ of strict 2-monads as in
\Cref{def:colax-monad-morphism} induces a 2-functor
\begin{equation}
  \lambda^*_{\mathrm{ps}}\colon
  \PsAlg(\mathbb T)\longrightarrow\PsAlg(\mathbb R)
  \label{eq:pseudo-restriction-functor}
\end{equation}
over $V$.  It sends a pseudoaction
$(a,\iota,\chi)$ to the action $Va\,\lambda_X$ of
\Cref{eq:restricted-action}.  Its unit cell is $V\iota$, using
\Cref{eq:general-monad-unit}.  Put
$c_X:=Va\lambda_X$ and
$\Lambda_X:=\lambda_{\mathbb T X}\mathbb R\lambda_X$.  Its
multiplication cell is the single composite
\begin{equation}
c_X\mathbb Rc_X
=Va\,V(\mathbb Ta)\Lambda_X
\xRightarrow{\ V\chi\,\Lambda_X\ }
Va\,V\mu^{\mathbb T}_X\Lambda_X
=c_X\mu^{\mathbb R}_{VX}.
 \label{eq:restricted-pseudo-multiplication}
\end{equation}
For a pseudomorphism
$f\colon(X,a)\to(Y,b)$ with comparison
$\overline f\colon b\mathbb T f\Rightarrow fa$, the restricted
comparison is
\begin{equation}
Vb\,\lambda_Y\,\mathbb R(Vf)
=Vb\,V(\mathbb T f)\lambda_X
\xRightarrow{\ V\overline f\,\lambda_X\ }
Vf\,Va\,\lambda_X.
\label{eq:restricted-pseudomorphism}
\end{equation}
Algebra transformations are sent by $V$.  Normal pseudoalgebras are
sent to normal pseudoalgebras.
\end{proposition}

\begin{proof}
The first equality in
\Cref{eq:restricted-pseudo-multiplication} is 2-naturality of
$\lambda$ at $a$, and the second is
\Cref{eq:general-monad-multiplication}.  Pasting the two pseudoalgebra
unit diagrams for $(a,\iota,\chi)$ with
\Cref{eq:general-monad-unit} gives the two unit diagrams for the
restricted action.  For associativity, paste the associativity diagram
for $\chi$ with the two instances of the colax multiplication square
\begin{equation*}
\begin{tikzcd}[column sep=large,row sep=large]
\mathbb R^2VX
  \arrow[r,"\mathbb R\lambda_X"]
  \arrow[d,"\mu^{\mathbb R}_{VX}"']
& \mathbb R V\mathbb TX
  \arrow[r,"\lambda_{\mathbb TX}"]
& V\mathbb T^2X
  \arrow[d,"V\mu^{\mathbb T}_X"] \\
\mathbb RVX \arrow[rr,"\lambda_X"']
&& V\mathbb TX .
\end{tikzcd}
\end{equation*}
The outside pastings are the two associativity composites for
\Cref{eq:restricted-pseudo-multiplication}; hence they agree.

The unit axiom for a pseudomorphism follows by the same pasting
argument.  Its multiplication axiom follows similarly, using
2-naturality of $\lambda$ at $f$.  If
$\alpha\colon(f,\overline f)\Rightarrow(g,\overline g)$ is an algebra
transformation, its restricted compatibility equation is the
calculation
\begin{align*}
  (V\alpha\,Va\lambda_X)(V\overline f\,\lambda_X)
  &=V((\alpha a)\overline f)\lambda_X\\
  &=V(\overline g(b\mathbb T\alpha))\lambda_X\\
  &=(V\overline g\,\lambda_X)
    (Vb\,\lambda_Y\mathbb R(V\alpha)).
\end{align*}
The last equality uses 2-naturality of $\lambda$ at the 2-cell
$\alpha$, namely
$V(\mathbb T\alpha)\lambda_X=\lambda_Y\mathbb R(V\alpha)$.
These assignments respect
both compositions because $V$ is strict and the displayed cells are
defined by whiskering and pasting.  If $\iota$ is an identity, then so
is $V\iota$, proving the last assertion.

When $V$ is an identity 2-functor, this is the restriction-of-scalars
construction of \citet[Theorem~A.7]{PerroneTholen2022}.  The proof here
includes the relative case needed below, where $V=\Dom$.
\end{proof}

\subsection{The comma-to-squaring comparison}

For every $P\colon E\to B$, define a functor
\begin{equation}
\kappa_P\colon E^{\mathbf 2}\longrightarrow P\!\downarrow\!B
  \label{eq:kappa}
\end{equation}
by
\begin{align}
  \kappa_P(f\colon x\to y)&=(x,Pf),
  \label{eq:kappa-object}\\
  \kappa_P(r,s)&=(r,Ps)
  \label{eq:kappa-arrow}
\end{align}
for a commutative square $(r,s)\colon f\to g$.  The comma condition
for \Cref{eq:kappa-arrow} is
\begin{equation*}
Ps\,Pf=Pg\,Pr,
\end{equation*}
which is the image under $P$ of the square equation $sf=gr$.

\begin{samepage}
\begin{proposition}\label{prop:kappa-monad-comparison}
The functors $\kappa_P$ are the components of a 2-natural
transformation
\begin{equation}
\kappa\colon\Sq\Dom\Longrightarrow\Dom\C.
  \label{eq:kappa-transformation}
\end{equation}
They satisfy the unit and multiplication equations
\begin{align}
  \kappa_P I_E&=H^{P},
  \label{eq:kappa-unit}\\
  \kappa_P m_E
  &=M^{P}\,\kappa_{\C P}\,(\kappa_P)^{\mathbf 2}.
  \label{eq:kappa-multiplication}
\end{align}
Thus $(\Dom,\kappa)$ is a colax morphism from $\C$ to $\Sq$ in the
sense of \Cref{def:colax-monad-morphism}.  It therefore pulls strict
$\C$-algebras back to strict $\Sq$-algebras.
\end{proposition}
\end{samepage}

\begin{proof}
For a square $(T,S)\colon P\to Q$, one has
\begin{equation*}
\overline T\,\kappa_P=\kappa_Q T^{\mathbf 2}
\end{equation*}
as displayed by
\begin{equation*}
\begin{tikzcd}[column sep=huge,row sep=large]
E^{\mathbf 2} \arrow[r,"\kappa_P"] \arrow[d,"T^{\mathbf 2}"']
& P\!\downarrow\!B \arrow[d,"\overline T"] \\
F^{\mathbf 2} \arrow[r,"\kappa_Q"']
& Q\!\downarrow\!C .
\end{tikzcd}
\end{equation*}
Both sides send $f\colon x\to y$ to
$(Tx,Q(Tf))=(Tx,S(Pf))$.  For a 2-cell
$(\beta,\alpha)$, both induced transformations have component on $f$
given by $(\beta_x,\alpha_{Py})$.  This proves 2-naturality.

The unit identity in \Cref{eq:kappa-unit} is the commutativity of
\begin{equation*}
\begin{tikzcd}[column sep=huge,row sep=large]
E \arrow[r,"I_E"] \arrow[dr,"H^P"']
& E^{\mathbf 2} \arrow[d,"\kappa_P"] \\
& P\!\downarrow\!B .
\end{tikzcd}
\end{equation*}
Indeed, the image of $1_x$ is $(x,1_{Px})$, and the image of a square
between identities is $(f,Pf)$.  The multiplication identity is the
commutativity of
\begin{equation}
\begin{tikzcd}[column sep=huge,row sep=large]
(E^{\mathbf 2})^{\mathbf 2}
  \arrow[r,"{\kappa_{\C P}(\kappa_P)^{\mathbf 2}}"]
  \arrow[d,"m_E"']
& \C P\!\downarrow\!B
  \arrow[d,"M^P"] \\
E^{\mathbf 2} \arrow[r,"\kappa_P"']
& P\!\downarrow\!B .
\end{tikzcd}
\label{eq:kappa-multiplication-diagram}
\end{equation}
To verify \Cref{eq:kappa-multiplication-diagram}, apply both sides
to the square in \Cref{eq:square-multiplication}.  The left side is
\begin{equation*}
(x,P(sf)).
\end{equation*}
On the right, $(\kappa_P)^{\mathbf 2}$ first regards the square as
the comma arrow
\begin{equation*}
(r,Ps)\colon(x,Pf)\longrightarrow(x',Pg).
\end{equation*}
Then $\kappa_{\C P}$ regards that arrow as the object
$((x,Pf),Ps)$ of $\C^2P$, and $M^{P}$ sends it to
$(x,Ps\,Pf)=(x,P(sf))$.  It remains to consider a morphism between
commutative squares.  Let
\begin{equation*}
  (p,q)\colon f\to f_1,
  \qquad
  (p',q')\colon g\to g_1
\end{equation*}
be a morphism from $(r,s)\colon f\to g$ to
$(r_1,s_1)\colon f_1\to g_1$ in $(E^{\mathbf 2})^{\mathbf 2}$.
Thus $p'r=r_1p$ and $q's=s_1q$.  Multiplication in the squaring monad
sends this morphism to the square $(p,q')$ between the two diagonal
arrows, and the left side of \Cref{eq:kappa-multiplication} sends it
to the comma arrow
\begin{equation*}
  (p,Pq').
\end{equation*}
On the right, $(\kappa_P)^{\mathbf 2}$ gives the pair of comma arrows
$(p,Pq)$ and $(p',Pq')$; then $\kappa_{\C P}$ takes the first as its
inner total component and $Pq'$ as its outer base component.  Finally,
$M^{P}$ retains precisely $p$ and $Pq'$.  Hence both sides agree on
arrows, and the functors in \Cref{eq:kappa-multiplication} are equal.
\end{proof}

\begin{proposition}\label{prop:action-to-square-action}
If $(P,A)$ is a strict $\C$-algebra, then
\begin{equation}
F^{P}:=A\kappa_P\colon E^{\mathbf 2}\longrightarrow E
  \label{eq:FP}
\end{equation}
is a strict $\Sq$-algebra.  This construction is compatible with
strict algebra morphisms and algebra 2-cells.  More precisely, the
assignments define a strict 2-functor
\begin{equation}
  \kappa^{*}\colon\Alg(\C)\longrightarrow\Alg(\Sq)
  \label{eq:kappa-pullback}
\end{equation}
satisfying $U_{\Sq}\kappa^{*}=\Dom U_{\C}$ as an equality of 2-functors.
\end{proposition}

\begin{proof}
The unit equation follows from
\begin{equation*}
F^{P}I_E=A\kappa_PI_E=AH^{P}=1_E.
\end{equation*}
For multiplication, use \Cref{eq:kappa-multiplication}, the
$\C$-algebra law, and 2-naturality of $\kappa$ at the square
$(A,1_B)\colon\C P\to P$:
\begin{align*}
  F^{P}m_E
  &=A\kappa_Pm_E\\
  &=AM^{P}\kappa_{\C P}(\kappa_P)^{\mathbf 2}\\
  &=A\overline A\kappa_{\C P}(\kappa_P)^{\mathbf 2}\\
  &=A\kappa_P A^{\mathbf 2}(\kappa_P)^{\mathbf 2}\\
  &=F^{P}(F^{P})^{\mathbf 2}.
\end{align*}
This proves \Cref{eq:square-algebra-laws}.  Now let
$(T,S)\colon(P,A)\to(Q,D)$ be a strict $\C$-morphism.  Then
\begin{align*}
  TF^{P}
  &=TA\kappa_P
   =D\overline T\kappa_P
   =D\kappa_QT^{\mathbf 2}
   =F^{Q}T^{\mathbf 2},
\end{align*}
where the third equality is 2-naturality of $\kappa$.  Hence $T$ is a
strict $\Sq$-morphism.  If $(\beta,\alpha)$ is an algebra 2-cell, then
\begin{align*}
  \beta F^{P}
  &=\beta A\kappa_P
   =D\overline\beta\kappa_P
   =D\kappa_Q\beta^{\mathbf 2}
   =F^{Q}\beta^{\mathbf 2}.
\end{align*}
This is the $\Sq$-algebra 2-cell equation.
Since $\kappa^{*}$ leaves the underlying functors and natural
transformations unchanged, it preserves both kinds of composition and
identities strictly.  The displayed equality of forgetful 2-functors
is immediate from the construction.
\end{proof}

\begin{theorem}[The Grothendieck factorization 2-functor]
\label{thm:grothendieck-factorization-functor}
The colax morphism of 2-monads $(\Dom,\kappa)$ and the comparison
isomorphism of \Cref{thm:main} determine a strict 2-functor
\begin{equation}
  \Fact:=\kappa^{*}\Gamma\colon
  \SCoFib\longrightarrow\Alg(\Sq).
  \label{eq:Fact-definition}
\end{equation}
\begin{samepage}
The following triangle and square commute strictly:
\begin{equation*}
\begin{tikzcd}[column sep=large,row sep=large]
\SCoFib \arrow[rr,"\Gamma"] \arrow[dr,"\Fact"']
  && \Alg(\C) \arrow[dl,"\kappa^{*}"] \\
  & \Alg(\Sq), &
\end{tikzcd}
\qquad
\begin{tikzcd}[column sep=large,row sep=large]
\SCoFib \arrow[r,"\Fact"] \arrow[d,"U"']
  & \Alg(\Sq) \arrow[d,"U_{\Sq}"] \\
\CAT^{\mathbf 2} \arrow[r,"\Dom"'] & \CAT .
\end{tikzcd}
\end{equation*}
\end{samepage}
More explicitly, if $P\colon E\to B$ is a split cofibration, then
$\Fact(P)$ has action
\begin{equation}
  F^{P}:=L^{P}\kappa_P\colon E^{\mathbf 2}\longrightarrow E.
  \label{eq:Fact-action}
\end{equation}
For $f\colon x\to y$ and a commutative square
$(r,s)\colon f\to g$, where $g\colon x'\to y'$, one has
\begin{align}
  F^{P}(f)&=(Pf)_{!}x,
  \label{eq:FP-object}\\
  P F^{P}(r,s)&=Ps,
  \notag\\
  F^{P}(r,s)\,\delta^{x}_{Pf}&=\delta^{x'}_{Pg}\,r.
  \label{eq:Fact-on-squares}
\end{align}
Moreover,
\begin{equation}
\begin{tikzcd}[column sep=large,row sep=large]
x \arrow[r,"{\delta^{x}_{Pf}}"] \arrow[d,"r"']
  & (Pf)_{!}x \arrow[r,"\nu_f"]
      \arrow[d,"{F^{P}(r,s)}"]
  & y \arrow[d,"s"] \\
x' \arrow[r,"{\delta^{x'}_{Pg}}"']
  & (Pg)_{!}x' \arrow[r,"\nu_g"']
  & y'
\end{tikzcd}
\label{eq:Fact-square-of-factorizations}
\end{equation}
commutes.  The factorization selected by this algebra is
\begin{equation}
\begin{tikzcd}[column sep=large,row sep=large]
x \arrow[rr,"f"] \arrow[dr,"{\delta^{x}_{Pf}}"']
  && y \\
  & (Pf)_{!}x \arrow[ur,"\nu_f"']
\end{tikzcd}
\qquad
f=\nu_f\delta^{x}_{Pf},
\quad P\nu_f=1_{Py}.
  \label{eq:FP-factors}
\end{equation}

On a strictly cleavage-preserving square
$(T,S)\colon P\to Q$, the 2-functor $\Fact$ acts by the total functor
$T$, and
\begin{equation}
  T F^{P}=F^{Q}T^{\mathbf 2}.
  \label{eq:Fact-on-one-cells}
\end{equation}
In particular,
\begin{equation*}
  T\delta^{x}_{Pf}=\delta^{Tx}_{Q(Tf)},
  \qquad
  T\nu_f=\nu_{Tf}.
\end{equation*}
On a 2-cell
$(\beta,\alpha)\colon(T,S)\Rightarrow(T',S')$, it acts by $\beta$,
and
\begin{equation}
  \beta F^{P}=F^{Q}\beta^{\mathbf 2}.
  \label{eq:Fact-on-two-cells}
\end{equation}
Thus the passage from split Grothendieck transport to the chosen
factorizations of the total category is functorial in both
dimensions and is obtained by change of 2-monads.
\end{theorem}

\begin{proof}
By \Cref{thm:main}, $\Gamma$ is a strict isomorphism over
$\CAT^{\mathbf 2}$, and by
\Cref{prop:action-to-square-action}, $\kappa^{*}$ is a strict
2-functor satisfying $U_{\Sq}\kappa^{*}=\Dom U_{\C}$.  Their composite
therefore gives \Cref{eq:Fact-definition}; both displayed diagrams
commute strictly, and $U_{\Sq}\Fact=\Dom U$.

For a split cofibration, $\Gamma(P)=(P,L^{P})$, so its induced
$\Sq$-action is $F^{P}=L^{P}\kappa_P$.  Since
$\kappa_P(f)=(x,Pf)$, \Cref{eq:FP-object} follows.  If
$(r,s)\colon f\to g$, then $\kappa_P(r,s)=(r,Ps)$, and the defining
property of $L^{P}$ gives \Cref{eq:Fact-on-squares}.  The left square
in \Cref{eq:Fact-square-of-factorizations} is that same equation.
For the right square, both $\nu_gF^{P}(r,s)$ and $s\nu_f$ become
$gr=sf$ after precomposition with $\delta^{x}_{Pf}$; cocartesian
uniqueness makes them equal.

The left and right factors selected by a $\Sq$-action are respectively
$F^{P}(1_x,f)$ and $F^{P}(f,1_y)$.  Now
\begin{equation*}
  \kappa_P(1_x,f)=(1_x,Pf),
  \qquad
  \kappa_P(f,1_y)=(f,1_{Py}).
\end{equation*}
Applying $L^{P}$ gives $\delta^{x}_{Pf}$ and $\nu_f$, respectively,
which proves \Cref{eq:FP-factors}.

\Cref{eq:Fact-on-one-cells} and
\Cref{eq:Fact-on-two-cells} are the strict algebra-morphism and
algebra-2-cell equations supplied by $\kappa^{*}$.  The equality
$T\nu_f=\nu_{Tf}$ may also be seen directly: $T\nu_f$ is vertical and
\begin{equation*}
  (T\nu_f)\delta^{Tx}_{Q(Tf)}
  =(T\nu_f)(T\delta^{x}_{Pf})=Tf,
\end{equation*}
so cocartesian uniqueness identifies it with $\nu_{Tf}$.
\end{proof}

\begin{theorem}[The pseudoalgebraic Grothendieck factorization]
\label{thm:pseudo-grothendieck-factorization}
Restriction of scalars along the colax monad morphism
$(\Dom,\kappa)$ and the comparison of
\Cref{thm:normal-pseudo} determine a 2-functor
\begin{equation}
  \Fact_{\mathrm n}:=
  \kappa^*_{\mathrm{ps}}\Gamma_{\mathrm n}\colon
  \NClCoFib\longrightarrow\PsAlg_{\mathrm n}(\Sq).
  \label{eq:pseudo-Fact-definition}
\end{equation}
For a normally cloven cofibration $P\colon E\to B$, its action is
\begin{equation*}
  F^P=L^P\kappa_P\colon E^{\mathbf 2}\longrightarrow E.
\end{equation*}
It sends $f\colon x\to y$ to $(Pf)_!x$ and selects the normalized
factorization
\begin{equation}
\begin{tikzcd}[column sep=large,row sep=large]
x \arrow[rr,"f"] \arrow[dr,"\delta^x_{Pf}"']
&&y\\
&(Pf)_!x\arrow[ur,"\nu_f"']&
\end{tikzcd}
\qquad f=\nu_f\delta^x_{Pf}.
\label{eq:pseudo-Fact-factors}
\end{equation}
Under the isomorphism in
\Cref{eq:sq-pseudo-OFS}, this pseudoalgebra is the orthogonal
factorization system
\begin{equation}
  \bigl(\CoCart(P),P^{-1}\Iso(B)\bigr)
  \label{eq:pseudo-Fact-OFS}
\end{equation}
equipped with the chosen factorizations displayed in
\Cref{eq:pseudo-Fact-factors}.  On a 1-cell
$(T,S)\colon P\to Q$ of $\NClCoFib$, the pseudomorphism comparison at
$f\colon x\to y$ is the unique vertical isomorphism
\begin{equation}
  \theta_f\colon
  (Q(Tf))_!Tx=(S(Pf))_!Tx
  \xrightarrow{\ \cong\ }T((Pf)_!x)
  \label{eq:pseudo-Fact-morphism}
\end{equation}
satisfying
\begin{equation}
\begin{tikzcd}[column sep=huge,row sep=large]
Tx \arrow[r,"{\delta^{Tx}_{Q(Tf)}}"]
  \arrow[dr,"{T\delta^x_{Pf}}"']
  & (Q(Tf))_!Tx \arrow[d,"{\theta_f}","{\cong}"'] \\
  & T((Pf)_!x).
\end{tikzcd}
\label{eq:pseudo-Fact-comparison-triangle}
\end{equation}
The upper row of \Cref{eq:intro-architecture} is the restriction of
this construction to split cleavages and strictly
cleavage-preserving squares.
\end{theorem}

\begin{proof}
\Cref{prop:pseudo-restriction-scalars,prop:kappa-monad-comparison}
give a 2-functor
\begin{equation*}
  \kappa^*_{\mathrm{ps}}\colon
  \PsAlg_{\mathrm n}(\C)\longrightarrow
  \PsAlg_{\mathrm n}(\Sq),
\end{equation*}
and \Cref{thm:normal-pseudo} gives $\Gamma_{\mathrm n}$, proving
\Cref{eq:pseudo-Fact-definition}.  The restriction formula in
\Cref{eq:restricted-action} gives the action $L^P\kappa_P$.  The two factors selected
by a normal pseudo-$\Sq$-algebra are
$F^P(1_x,f)$ and $F^P(f,1_y)$.  Since
\begin{equation*}
  \kappa_P(1_x,f)=(1_x,Pf),
  \qquad
  \kappa_P(f,1_y)=(f,1_{Py}),
\end{equation*}
they are $\delta^x_{Pf}$ and $\nu_f$, respectively.  Normality makes
the factorization of an identity the pair of identities.

By \Cref{eq:sq-pseudo-OFS}, the left class consists of the arrows $f$
for which $\nu_f$ is invertible.  If $\nu_f$ is invertible, then
$f=\nu_f\delta^x_{Pf}$ is cocartesian.  Conversely, suppose that $f$
is cocartesian.  Write $u=Pf$ and $\delta=\delta^x_u$.  Given
$g\colon u_!x\to z$ and $w\colon Py\to Pz$ with
$Pg=wP\nu_f=w$, the composite $g\delta$ lies over $wu$.
Cocartesianness of $f=\nu_f\delta$ gives a unique
$h\colon y\to z$ over $w$ such that $hf=g\delta$.  Hence
$(h\nu_f)\delta=g\delta$; cancellation of the cocartesian arrow
$\delta$ gives $h\nu_f=g$.  If $h'\nu_f=g$, then
$h'f=g\delta=hf$, so cocartesianness of $f$ gives $h'=h$.
Therefore $\nu_f$ is cocartesian.

A vertical cocartesian arrow $q\colon d\to e$ is invertible.  Its
universal property applied to $1_d$ gives a vertical arrow
$r\colon e\to d$ with $rq=1_d$.  The vertical arrows $qr$ and $1_e$
agree after precomposition with $q$, and cancellation gives
$qr=1_e$.  Thus $\nu_f$ is invertible and the left class is
$\CoCart(P)$.

The right class consists of the arrows $f$ for which
$\delta^x_{Pf}$ is invertible.  Applying $P$ shows that invertibility
of the lift implies invertibility of $Pf$.  Conversely, suppose
$u=Pf$ is invertible.  The identity $1_x$ lies over $u^{-1}u$, so
cocartesianness of $\delta^x_u$ gives a unique arrow
$r\colon u_!x\to x$ over $u^{-1}$ with
$r\delta^x_u=1_x$.  Both $\delta^x_ur$ and $1_{u_!x}$ lie over
$1_b$ and agree after $\delta^x_u$; cocartesian uniqueness gives
$\delta^x_ur=1_{u_!x}$.  Hence $\delta^x_u$ is invertible and the
right class is $P^{-1}\Iso(B)$.  This proves
\Cref{eq:pseudo-Fact-OFS}.

For a pseudomorphism of comma algebras, the restricted comparison in
\Cref{eq:restricted-pseudomorphism} is obtained by whiskering its
component with $\kappa_P$.  At $f$ this is the isomorphism in
\Cref{eq:pseudo-Fact-morphism}.  Its existence and uniqueness were
proved in \Cref{thm:normal-pseudo}.  Such a $T$ preserves the left
class because it preserves cocartesian arrows, and it preserves the
right class because $Q(Tf)=S(Pf)$ and functors preserve isomorphisms.
Algebra transformations are the total components of the
arrow-2-category 2-cells.  Finally, if the cleavage and its preservation
are strict, all transport compositors and comparison cells are
identities, so the construction restricts to
\Cref{thm:grothendieck-factorization-functor}.
\end{proof}

\begin{corollary}[Arbitrary cleavages and factorization]
\label{cor:fixed-base-pseudo-factorization}
For each category $B$, the componentwise formulas of
\Cref{prop:kappa-monad-comparison} restrict to a colax monad morphism
from $\C_B$ on $\CAT/B$ to $\Sq$ on $\CAT$, with underlying 2-functor
$\Dom\colon\CAT/B\to\CAT$.
Restriction of scalars along this fixed-base form of $\kappa$ and the
isomorphism of
\Cref{thm:fixed-base-pseudo} give a 2-functor
\begin{equation}
  \Fact^{\mathrm{ps}}_B\colon
  \ClCoFib(B)\longrightarrow\PsAlg(\Sq).
  \label{eq:fixed-base-pseudo-Fact}
\end{equation}
For a cloven cofibration $P\colon E\to B$, its action is
$F^P=L^P\kappa_P$.  If
\begin{equation*}
  \jmath_x:=\delta^x_{1_{Px}}\colon
  x\xrightarrow{\ \cong\ }F^P(1_x)
\end{equation*}
denotes its pseudo-$\Sq$-unit, the factors selected from the
pseudoaction are
\begin{align}
  e_f&:=F^P(1_x,f)\jmath_x=\delta^x_{Pf},
  \label{eq:nonnormal-left-factor}\\
  m_f&:=\jmath_y^{-1}F^P(f,1_y)=\nu_f.
  \label{eq:nonnormal-right-factor}
\end{align}
They determine the orthogonal factorization system
\begin{equation}
  \bigl(\CoCart(P),P^{-1}\Iso(B)\bigr).
  \label{eq:fixed-base-pseudo-OFS}
\end{equation}
The chosen factorization of $1_x$ is
$x\xrightarrow{\jmath_x}F^P(1_x)
\xrightarrow{\jmath_x^{-1}}x$; it is the identity--identity
factorization precisely when the cleavage is normal.
\end{corollary}

\begin{proof}
The fixed-base restriction of
\Cref{prop:pseudo-restriction-scalars,prop:kappa-monad-comparison}
composed with \Cref{eq:fixed-base-pseudo-comparison} gives
\Cref{eq:fixed-base-pseudo-Fact}.  Naturality of the pseudo-$\Sq$-unit
at $f$ gives
\begin{equation*}
  F^P(f,1_y)F^P(1_x,f)\jmath_x=\jmath_y f;
\end{equation*}
hence $f=m_fe_f$.  Since
\begin{equation*}
  \kappa_P(1_x,f)=(1_x,Pf),
  \qquad
  \kappa_P(f,1_y)=(f,1_{Py}),
\end{equation*}
the defining equations for $L^P$ give
\begin{equation*}
  F^P(1_x,f)\jmath_x=\delta^x_{Pf},
  \qquad
  \jmath_y^{-1}F^P(f,1_y)\delta^x_{Pf}=f.
\end{equation*}
The second equation, together with verticality, identifies
$\jmath_y^{-1}F^P(f,1_y)$ with $\nu_f$, proving
\Cref{eq:nonnormal-left-factor,eq:nonnormal-right-factor}.

For the chosen factorization $f=m_fe_f$, set
\begin{equation*}
  \mathcal E_F=\{f\mid m_f\text{ is invertible}\},
  \qquad
  \mathcal M_F=\{f\mid e_f\text{ is invertible}\}.
\end{equation*}
Then $\mathcal E_F=\CoCart(P)$ and
$\mathcal M_F=P^{-1}\Iso(B)$.  Indeed, these are the same cancellation
and inverse-lift arguments used in the proof of
\Cref{thm:pseudo-grothendieck-factorization}; they do not use
normality.  Replacing only the designated lifts over identity arrows
by identities gives a normal cleavage with the same cocartesian arrows.
\Cref{thm:pseudo-grothendieck-factorization} then shows that this
cleavage-independent pair is an orthogonal factorization system.  This
proves \Cref{eq:fixed-base-pseudo-OFS}.  Substituting
$f=1_x$ gives the final assertion.
\end{proof}

\subsection{The resulting strict and orthogonal factorization systems}

Recall that a \emph{strict factorization system} on $E$ is a pair
$(\mathcal E,\mathcal M)$ of wide subcategories such that every arrow
$f$ has a unique expression $f=me$ with $e\in\mathcal E$ and
$m\in\mathcal M$.  Unlike an ordinary orthogonal factorization
system, neither subcategory is required to contain every isomorphism.
For a strict factorization algebra $F$, with factors as in
\Cref{eq:sq-factorization}, the corresponding strict classes are
\begin{equation}
  \mathcal E_F=\{f\mid m_f=1\},
  \qquad
  \mathcal M_F=\{f\mid e_f=1\}.
  \label{eq:strict-classes}
\end{equation}
We use the factorization-and-lifting definition of an ordinary
orthogonal factorization system: both classes are closed under
composition with isomorphisms, every arrow factors as an arrow of the
left class followed by an arrow of the right class, and every square
from a left-class arrow to a right-class arrow has a unique diagonal.
These axioms imply
$\mathcal E={}^{\perp}\mathcal M$ and
$\mathcal M=\mathcal E^{\perp}$; in particular, both classes contain
all isomorphisms and are closed under composition.
For a strict factorization algebra with factors $f=m_fe_f$, its
associated ordinary classes are
\begin{equation}
  \mathcal E^{\simeq}=\{f\mid m_f\text{ is invertible}\},
  \qquad
  \mathcal M^{\simeq}=\{f\mid e_f\text{ is invertible}\}.
  \label{eq:ordinary-classes}
\end{equation}
For a functor $P\colon E\to B$, we write
\begin{equation*}
  P^{-1}\Iso(B)
  =\{f\in E^{\mathbf 2}\mid Pf\text{ is invertible in }B\}.
\end{equation*}
The general theory of strict factorization algebras shows that these
classes form an orthogonal factorization system.  We verify that claim
directly for the algebra at hand.

\begin{theorem}[Strict and orthogonal factorizations]
\label{thm:factorization-comparison}
Let $P\colon E\to B$ be a split cofibration.
\begin{enumerate}[label=\textup{(\roman*)},leftmargin=2.2em]
\item The strict $\Sq$-algebra $F^{P}$ corresponds to the strict
      factorization system
      \begin{equation}
        (\Delta_P,\Vrt(P)),
        \qquad
        \Delta_P=
        \{\delta^{x}_{u}\mid x\in E, u\colon Px\to b\}.
        \label{eq:strict-fs}
\end{equation}
\item The ordinary orthogonal factorization system associated with
      \Cref{eq:strict-fs} is
      \begin{equation}
        \bigl(\CoCart(P),P^{-1}\Iso(B)\bigr).
        \label{eq:ordinary-fs}
\end{equation}
\end{enumerate}
\end{theorem}

\begin{proof}
The split equations show that $\Delta_P$ contains all identities and
is closed under composition.  The vertical arrows form a wide
subcategory.  The factorization in \Cref{eq:cofib-factorization} writes every
arrow as an element of $\Delta_P$ followed by a vertical arrow.  If
$f=me$ is any such factorization and $x$ is the domain of $f$, then
$e\in\Delta_P$ has the form $\delta^{x}_{u}$.  Since $m$ is vertical,
$u=Pe=Pf$, and hence $e=\delta^{x}_{Pf}$.  The equality
$m\delta^{x}_{Pf}=f$ then forces $m=\nu_f$ by cocartesian uniqueness.
Thus the factorization is unique.  \Cref{eq:FP-factors} shows that the
classes defined in \Cref{eq:strict-classes} for $F^{P}$ are respectively $\Delta_P$ and
$\Vrt(P)$, and therefore identifies this strict system with the one
derived from $F^{P}$.

We prove the second assertion directly.  First, the enlargement of
$\Delta_P$ prescribed by a strict factorization algebra consists of
the arrows $f$ for which the right factor $\nu_f$ is invertible.  Such
an $f$ is $P$-cocartesian, because it is an isomorphism after a
cocartesian arrow.  Conversely, if $f=\nu_f\delta^{x}_{Pf}$ is
cocartesian, then $\nu_f$ is cocartesian by right cancellation for
cocartesian arrows.  Here is the cancellation argument.  If
$e\colon x\to y$ and $me\colon x\to z$ are cocartesian, and if
$g\colon y\to t$ satisfies $Pg=wPm$, cocartesianness of $me$ gives a
unique $k\colon z\to t$ over $w$ such that $kme=ge$.
Cocartesianness of $e$ then gives $km=g$, and uniqueness for $me$
gives uniqueness of $k$; hence $m$ is cocartesian.  A vertical
cocartesian arrow $m\colon y\to z$ is an isomorphism: applying its
universal property to $1_y$ gives an arrow $r\colon z\to y$ with
$rm=1_y$, and cocartesian uniqueness applied to
$(mr)m=m=1_zm$ gives $mr=1_z$.  Thus the enlarged left class is
$\CoCart(P)$.

The enlarged right class consists of those $f$ for which
$\delta^{x}_{Pf}$ is invertible.  A chosen lift $\delta^{x}_{u}$ is
invertible if and only if $u$ is invertible.  One implication follows
by applying $P$.  For the other, the splitting gives the strict object
equalities needed for
\begin{equation*}
  \delta^{u_{!}x}_{u^{-1}}\delta^{x}_{u}
    =\delta^{x}_{u^{-1}u}=1_x,
  \qquad
  \delta^{x}_{u}\delta^{u_{!}x}_{u^{-1}}
    =\delta^{u_{!}x}_{uu^{-1}}=1_{u_{!}x}.
\end{equation*}
Hence the enlarged right class is $P^{-1}\Iso(B)$.

Consider a commutative square
\begin{equation}
\begin{tikzcd}[column sep=large,row sep=large]
x \arrow[r,"r"] \arrow[d,"e"']
  & z \arrow[d,"m"] \\
y \arrow[r,"s"'] \arrow[ur,dashed,"d" description]
  & t
\end{tikzcd}
\label{eq:orth-square}
\end{equation}
with $e$ $P$-cocartesian and $Pm$ invertible.  Put
\begin{equation*}
q=(Pm)^{-1}Ps\colon Py\longrightarrow Pz.
\end{equation*}
The square equation gives $qPe=Pr$, so cocartesianness provides a
unique $d\colon y\to z$ over $q$ with $de=r$.  Now $md$ and $s$ lie
over the same base arrow and agree after $e$; cocartesian uniqueness
gives $md=s$.  Any diagonal in \Cref{eq:orth-square} is forced to lie
over $q$, so it equals $d$.  The two classes are closed under
composition with isomorphisms: isomorphisms are $P$-cocartesian and
cocartesian arrows compose, while $P$ preserves isomorphisms.  Finally,
every arrow has its cocartesian--vertical factorization, whose right
part belongs to $P^{-1}\Iso(B)$.  The defining axioms of an orthogonal
factorization system are therefore satisfied.
\end{proof}

Thus Part~\textup{(ii)} recovers, by duality, the prefibrational
factorization of \citet[Section~3.7]{RosickyTholen2007}, while
Part~\textup{(i)}
records the additional strict data carried by the chosen split
cleavage and transported by the 2-functor $\Fact$.

The strict system in \Cref{eq:strict-fs} depends on the designated
split cleavage.  Its associated ordinary orthogonal factorization
system in \Cref{eq:ordinary-fs} depends only on the underlying
cofibration $P$.

\begin{corollary}[Grothendieck constructions]
\label{cor:grothendieck-construction-factorization}
Let $X\colon B\to\CAT$ be a strict functor, and let
$\pi_X\colon\int_BX\to B$ be its Grothendieck construction.  Thus an
object is a pair $(b,x)$ with $x\in Xb$, and a morphism
\begin{equation*}
  (u,\varphi)\colon(b,x)\longrightarrow(c,y)
\end{equation*}
consists of $u\colon b\to c$ and
$\varphi\colon X(u)x\to y$ in $Xc$.  The canonical split cleavage has
designated cocartesian arrows
\begin{equation*}
  (u,1_{X(u)x})\colon(b,x)\longrightarrow(c,X(u)x).
\end{equation*}
The algebra $\Fact(\pi_X)$ selects the factorization
\begin{equation}
\begin{tikzcd}[column sep=huge]
(b,x) \arrow[r,"{(u,1_{X(u)x})}"]
  & (c,X(u)x) \arrow[r,"{(1_c,\varphi)}"]
  & (c,y).
\end{tikzcd}
\label{eq:grothendieck-construction-factorization}
\end{equation}
Its strict factorization system is formed by the designated arrows
$(u,1)$ and the vertical arrows $(1_b,\psi)$.  Its associated ordinary
orthogonal factorization system is
\begin{equation*}
  \left(
    \{(u,\varphi)\mid\varphi\text{ is invertible}\},
    \{(u,\varphi)\mid u\text{ is invertible}\}
  \right).
\end{equation*}
\end{corollary}

\begin{proof}
The displayed cleavage is split because $X$ preserves identities and
composition strictly.  \Cref{eq:grothendieck-construction-factorization}
displays the cocartesian--vertical factorization of $(u,\varphi)$, so the assertion
about the strict system follows from
\Cref{thm:grothendieck-factorization-functor,thm:factorization-comparison}.
A morphism $(u,\varphi)$ is $\pi_X$-cocartesian precisely when
$\varphi$ is invertible.  Indeed, $(u,\varphi)$ is the composite of the
designated cocartesian arrow $(u,1)$ with the vertical arrow
$(1_c,\varphi)$.  Stability under composition and right cancellation
for cocartesian arrows show that the composite is cocartesian precisely
when the latter arrow is cocartesian, and a vertical arrow in the
Grothendieck construction is cocartesian precisely when its fibre
component is invertible.  The image of $(u,\varphi)$ under $\pi_X$ is
invertible precisely when $u$ is invertible.  The last assertion now
follows from \Cref{thm:factorization-comparison}.
\end{proof}

The normal form
\begin{equation*}
  (u,\varphi)=(1_c,\varphi)(u,1_{X(u)x})
\end{equation*}
is therefore not an auxiliary choice: it is the value of the restricted
$\Sq$-action on $(u,\varphi)$.  At the level of the Grothendieck
construction, the corresponding hierarchy of morphisms is recorded
explicitly by \citet[Proposition~4.1]{LucatelliNunesVakar2025}:
pseudomorphisms correspond to cartesian functors, strict morphisms to
cleavage-preserving functors, and strict indexed categories to split
fibrations.  The cofibrational formulation above is obtained by the
usual variance duality.

Grothendieck constructions of indexed categories also supply the total
categories in which forward- and reverse-mode CHAD for expressive total
languages are defined \citep{LucatelliNunesVakar2023}.

\begin{remark}\label{rem:right-skeletal}
The pair $(\CoCart(P),\Vrt(P))$ has factorizations and unique
diagonals against vertical arrows, but $\Vrt(P)$ need not contain
isomorphisms whose image under $P$ is a nonidentity isomorphism.  The
strict system in \Cref{eq:strict-fs} records the designated choices made by
the splitting; the ordinary system in \Cref{eq:ordinary-fs} records their
isomorphism closure.  Keeping these two levels separate removes the
need for a modified notion such as a ``right-skeletal orthogonal
factorization system.''
\end{remark}

\section{Duality and split fibrations}
\label{sec:dual}

\subsection{The dual comma 2-monad}

Taking opposites reverses natural transformations and defines a strict
2-isomorphism
\begin{equation}
  \mathcal O\colon(\CAT^{\mathbf 2})^{\mathrm{co}}
  \longrightarrow\CAT^{\mathbf 2},
  \qquad P\longmapsto P^{\mathrm{op}}.
  \label{eq:opposite-2-isomorphism}
\end{equation}
We use the standard involutive choice of opposites, so that taking
opposites twice is the identity.  For each $P$ there is also
a canonical isomorphism
\begin{equation}
  d_P\colon
  (P^{\mathrm{op}}\!\downarrow\!B^{\mathrm{op}})^{\mathrm{op}}
  \longrightarrow B\!\downarrow\!P,
  \label{eq:dual-comma-identification}
\end{equation}
which reorders an object to $(b,u,x)$.  An arrow from
$(x,u^{\mathrm{op}})$ to $(y,v^{\mathrm{op}})$ in the outer opposite
is represented before taking that opposite by
\begin{equation*}
  (f^{\mathrm{op}},s^{\mathrm{op}})
  \colon(y,v^{\mathrm{op}})\longrightarrow(x,u^{\mathrm{op}}).
\end{equation*}
The functor $d_P$ sends it to
\begin{equation*}
  (s,f)\colon(b,u,x)\longrightarrow(c,v,y).
\end{equation*}
Indeed, the comma equation before taking opposites becomes $Pf\,u=vs$.
The family $d=(d_P)_P$ is 2-natural, and every
$d_P$ is a strict isomorphism of categories over $B$.  We use this
family to present the transport of $\C$ across
\Cref{eq:opposite-2-isomorphism} as the dual comma 2-monad
$\mathbb F$.

Explicitly, for a square $(T,S)\colon P\to Q$, with $P\colon E\to B$
and $Q\colon F\to D$, let
$\widehat T\colon B\!\downarrow\!P\to D\!\downarrow\!Q$ send
$(b,u,x)$ to $(Sb,Su,Tx)$ and $(s,f)$ to $(Ss,Tf)$.  The naturality
square for $d$ is the strictly commutative square
\begin{equation}
\begin{tikzcd}[column sep=huge,row sep=large]
(P^{\mathrm{op}}\!\downarrow\!B^{\mathrm{op}})^{\mathrm{op}}
  \arrow[r,"{(\overline{T^{\mathrm{op}}})^{\mathrm{op}}}"]
  \arrow[d,"d_P"']
& (Q^{\mathrm{op}}\!\downarrow\!D^{\mathrm{op}})^{\mathrm{op}}
  \arrow[d,"d_Q"] \\
B\!\downarrow\!P
  \arrow[r,"{\widehat T}"']
& D\!\downarrow\!Q .
\end{tikzcd}
\label{eq:dual-comma-naturality}
\end{equation}
Both composites send an object to $(Sb,Su,Tx)$ and an arrow to
$(Ss,Tf)$.  For a 2-cell
$(\beta,\alpha)\colon(T,S)\Rightarrow(T',S')$, both induced
transformations have component $(\alpha_b,\beta_x)$ at $(b,u,x)$.
This proves 2-naturality on 1-cells and 2-cells without inserting any
coherence isomorphisms.

Explicitly, for
$P\colon E\to B$, put
\begin{equation*}
\mathbb F P=\Dom_{P}\colon
  B\!\downarrow\!P\longrightarrow B,
\end{equation*}
where an object is $(b,u,x)$ with $u\colon b\to Px$.  An arrow
\begin{equation*}
  (s,f)\colon(b,u,x)\longrightarrow(c,v,y)
\end{equation*}
consists of $s\colon b\to c$ and $f\colon x\to y$ satisfying
$Pf\,u=vs$, and $\Dom_{P}(s,f)=s$.  A square
$(T,S)\colon P\to Q$ is sent to the square whose total functor maps
$(b,u,x)$ to $(Sb,Su,Tx)$ and $(s,f)$ to $(Ss,Tf)$.  A 2-cell is sent
to the natural transformation whose component at $(b,u,x)$ is
$(\alpha_b,\beta_x)$; its comma equation is the naturality equation
$S'u\,\alpha_b=\alpha_{Px}Su=Q\beta_x\,Su$.

The unit has total component
\begin{equation*}
  J^{P}x=(Px,1_{Px},x),
  \qquad
  J^{P}f=(Pf,f),
\end{equation*}
and multiplication sends $(a,t,(b,u,x))$, with
$t\colon a\to b$ and $u\colon b\to Px$, to $(a,ut,x)$.
More explicitly, an arrow
\begin{equation*}
  (s,(r,f))\colon
  (a,t,(b,u,x))\longrightarrow(a',t',(b',u',x')),
\end{equation*}
with $rt=t's$ and $Pf\,u=u'r$, is sent by multiplication to
\begin{equation*}
  (s,f)\colon(a,ut,x)\longrightarrow(a',u't',x').
\end{equation*}
These are the opposite forms of
\Cref{eq:H-definition,eq:M-definition}, so the 2-naturality and monad
equations follow as strict equalities from
\Cref{prop:comma-monad}.

\begin{proposition}\label{prop:dual-comma-monad}
The displayed assignments define a strict 2-monad
$\mathbb F$ on $\CAT^{\mathbf 2}$.  Moreover, taking opposites induces
strict 2-isomorphisms
\begin{equation}
  \mathcal O_{\mathrm{cofib}}\colon
  \SCoFib^{\mathrm{co}}\longrightarrow\SFib,
  \qquad
  \mathcal O_{\mathrm{alg}}\colon
  \Alg(\C)^{\mathrm{co}}\longrightarrow\Alg(\mathbb F),
  \label{eq:dual-2-isomorphisms}
\end{equation}
where $\SFib$ denotes the 2-category of split fibrations, strictly
cleavage-preserving squares, and all 2-cells of $\CAT^{\mathbf 2}$
between such squares.
\end{proposition}

\begin{proof}
The opposite construction is involutive on categories and functors.
A 2-cell from $(T,S)$ to $(T',S')$ in
$(\CAT^{\mathbf 2})^{\mathrm{co}}$ is a 2-cell from $(T',S')$ to
$(T,S)$ in $\CAT^{\mathbf 2}$; taking componentwise opposites reverses
it once more and therefore gives a 2-cell from
$(T^{\mathrm{op}},S^{\mathrm{op}})$ to
$(T'^{\mathrm{op}},S'^{\mathrm{op}})$.  Thus \Cref{eq:opposite-2-isomorphism}
is a strict 2-isomorphism, with itself as inverse after the canonical
involutive identifications $(E^{\mathrm{op}})^{\mathrm{op}}=E$ fixed
above.

Transporting $\C$ and its unit and multiplication through
\Cref{eq:opposite-2-isomorphism} gives the monad
$\mathcal O\,\C^{\mathrm{co}}\mathcal O^{-1}$.  The 2-natural family
of strict isomorphisms in \Cref{eq:dual-comma-identification}
transports this monad to $\mathbb F$.  Direct application of opposites
gives the formulas for
$J^{P}$ and the multiplication displayed above.  Since $d_P$ and its
inverse are strict functors and all comparison squares commute
strictly, no pseudonatural coherence data enter the construction.
The chosen comma reordering is involutive and is strictly compatible
with objects, arrows, squares, and 2-cells.

A designated cocartesian arrow for $P^{\mathrm{op}}$ is, after taking
opposites, a designated cartesian arrow for $P$.  The split identity
and composition equations, strict preservation by a square, and the
equation defining a 2-cell are all transported componentwise.  This
gives $\mathcal O_{\mathrm{cofib}}$.  Transporting an action, a strict
algebra morphism, and an algebra 2-cell in the same way gives
$\mathcal O_{\mathrm{alg}}$.  Explicitly, the action transported from
$(P,A)$ to $P^{\mathrm{op}}$ is
\begin{equation*}
  \bigl(A^{\mathrm{op}}d_{P^{\mathrm{op}}}^{-1},
  1_{B^{\mathrm{op}}}\bigr)\colon
  \mathbb F(P^{\mathrm{op}})\longrightarrow P^{\mathrm{op}}.
\end{equation*}
In both cases the inverse is again formed
by taking opposites, so the two functors in
\Cref{eq:dual-2-isomorphisms} are strict 2-isomorphisms.
\end{proof}

\subsection{Split fibrations}

An arrow $e\colon z\to x$ is $P$-\emph{cartesian}
if every $g\colon y\to x$ and $v\colon Py\to Pz$ satisfying
$Pg=(Pe)v$ admit a unique $h\colon y\to z$ with
$eh=g$ and $Ph=v$.  A \emph{fibration} has such a lifting with
codomain $x$ for every $u\colon b\to Px$.  A split cleavage designates
one of them,
\begin{equation*}
  \delta^x_u\colon u^*x\longrightarrow x,
\end{equation*}
and satisfies strictly
\begin{equation*}
  (1_a)^*=1_{E_a},
  \qquad
  (uv)^*=v^*u^*,
  \qquad
  \delta^x_{1_{Px}}=1_x,
  \qquad
  \delta^x_{uv}=\delta^x_u\delta^{u^*x}_v.
\end{equation*}
Here $u\colon b\to Px$ and $v\colon c\to b$.  The 2-category
$\SFib$ has these split fibrations as objects, the squares satisfying
$T\delta^x_u=\delta^{Tx}_{Su}$ as 1-cells, and all arrow-2-category
2-cells between them.  Equivalently, these definitions are obtained
from \Cref{def:scofib} by the opposite construction of
\Cref{prop:dual-comma-monad}.

\begin{corollary}[Dual strict 2-monadicity]\label{cor:fibration}
The comparison 2-functor
\begin{equation*}
  \Gamma_{\mathrm f}\colon\SFib\longrightarrow\Alg(\mathbb F)
\end{equation*}
is an isomorphism of 2-categories over $\CAT^{\mathbf 2}$.  Hence split
Grothendieck fibrations are strictly 2-monadic when the base is allowed
to vary.
\end{corollary}

\begin{proof}
Taking co-2-categories in \Cref{thm:main} and conjugating by the two
strict 2-isomorphisms of \Cref{prop:dual-comma-monad} gives
\begin{equation*}
\begin{tikzcd}[column sep=huge,row sep=large]
\SCoFib^{\mathrm{co}} \arrow[r,"\Gamma^{\mathrm{co}}"]
  \arrow[d,"{\mathcal O_{\mathrm{cofib}}}"',"\cong"]
& \Alg(\C)^{\mathrm{co}}
  \arrow[d,"{\mathcal O_{\mathrm{alg}}}","\cong"'] \\
\SFib \arrow[r,"\Gamma_{\mathrm f}"']
& \Alg(\mathbb F).
\end{tikzcd}
\end{equation*}
Equivalently,
\begin{equation*}
  \Gamma_{\mathrm f}
  =\mathcal O_{\mathrm{alg}}\,
   \Gamma^{\mathrm{co}}\,
   \mathcal O_{\mathrm{cofib}}^{-1}.
\end{equation*}
The two vertical 2-isomorphisms in the displayed square lie over
$\mathcal O\colon(\CAT^{\mathbf 2})^{\mathrm{co}}
\to\CAT^{\mathbf 2}$.  Their base changes therefore cancel in the
conjugate, so $\Gamma_{\mathrm f}$ lies over the identity of
$\CAT^{\mathbf 2}$.  It is a strict 2-isomorphism.  To identify this
conjugate
with the usual fibration comparison, observe that the object
reconstruction in
\Cref{eq:delta-from-action} becomes
\begin{equation*}
  u^{*}x=A(b,u,x),
  \qquad
  \delta^{x}_{u}=A(u,1_x)\colon u^{*}x\longrightarrow x,
\end{equation*}
where $(u,1_x)\colon(b,u,x)\to J^{P}x$ is the canonical comma arrow.
These are the designated cartesian liftings of the transported split
cleavage.  Hence the conjugate is the comparison for split fibrations,
completing the proof.
\end{proof}

\section{Conclusion}

Street's formal theory of monads supplies the Eilenberg--Moore language
in which a cleavage, its coherence, the preservation law for squares,
and the compatibility of transformations form one 2-dimensional
algebraic structure \citep{Street1972}.  His lax-idempotent treatment of
fibrations explains the adjunction between an action and the unit and
the passage from strict to coherent transport \citep{Street1980}.  The
comparison $\kappa$ adds the step developed here: restriction of scalars
turns that transport action into the factorization action on the total
category.  Consequently the cocartesian--vertical factorization and its
orthogonal closure are determined coherently by the Grothendieck
transport data.

The construction has been carried out here in $\CAT$.  A natural next
question is whether the comma-to-squaring comparison admits versions
for genuinely 2-dimensional fibrations, in the sense developed by
\citet{Hermida1999} and \citet{Buckley2014}; see also the related global
perspective of \citet[Section~9]{PeschkeTholen2020}.  The lax comma
2-categories of \citet{ClementinoLucatelliNunes2024} provide another
possible setting for such an extension.  A second question
concerns fibrations carrying additional monoidal or enriched structure.
Monoidal Grothendieck constructions are studied by
\citet{MoellerVasilakopoulou2020}, while sufficient conditions for
monoidal closure of their total categories are developed by
\citet{LucatelliNunesVakar2025}; for enriched fibrations, see
\citet{Vasilakopoulou2018}.  Such extensions should clarify which
factorization structures on the corresponding total 2-categories or
enriched categories are controlled by higher or enriched transport.
We leave these questions for future work.

\section*{Acknowledgements}
This work grew out of the first author's 2023 Fields Research
Fellowship, \emph{Grothendieck Descent Theory, Fibrations, and
Factorizations}, hosted by the second author.  We thank the Fields
Institute for Research in Mathematical Sciences for its support and
hospitality.

\section*{Funding}
The first author was supported by the Fields
Institute for Research in Mathematical Sciences through a Fields
Research Fellowship in 2023, and by the Centre for Mathematics of the
University of Coimbra (CMUC) under the Funda\c{c}\~ao para a Ci\^encia
e a Tecnologia (FCT), through the grants UID/00324/2025 and
UID/PRR/00324/2025.  Part of the work was carried out while the first
author was supported by the Deutsche Forschungsgemeinschaft (DFG,
German Research Foundation) through the project \emph{Higher-Order
Monad-based Programming and Reasoning} (HOMBRe), project number
501369690, led by Sergey Goncharov, during his postdoctoral appointment
at the School of Computer Science, University of Birmingham.
The funders had no role in the preparation of the manuscript.

\bibliographystyle{plainnat}
\bibliography{references}

\begin{thebibliography}{32}
\providecommand{\natexlab}[1]{#1}
\providecommand{\url}[1]{\texttt{#1}}
\expandafter\ifx\csname urlstyle\endcsname\relax
  \providecommand{\doi}[1]{doi: #1}\else
  \providecommand{\doi}{doi: \begingroup \urlstyle{rm}\Url}\fi

\bibitem[Blackwell et~al.(1989)Blackwell, Kelly, and
  Power]{BlackwellKellyPower1989}
Robert Blackwell, G.~Max Kelly, and A.~John Power.
\newblock Two-dimensional monad theory.
\newblock \emph{J. Pure Appl. Algebra}, 59\penalty0 (1):\penalty0 1--41, 1989.
\newblock \doi{10.1016/0022-4049(89)90160-6}.

\bibitem[Buckley(2014)]{Buckley2014}
Mitchell Buckley.
\newblock Fibred 2-categories and bicategories.
\newblock \emph{J. Pure Appl. Algebra}, 218\penalty0 (6):\penalty0 1034--1074,
  2014.
\newblock \doi{10.1016/j.jpaa.2013.11.002}.
\newblock URL \url{https://arxiv.org/abs/1212.6283}.

\bibitem[Clementino and L{\'o}pez~Franco(2016)]{ClementinoLopezFranco2016}
Maria~Manuel Clementino and Ignacio L{\'o}pez~Franco.
\newblock Lax orthogonal factorisation systems.
\newblock \emph{Adv. Math.}, 302:\penalty0 458--528, 2016.
\newblock \doi{10.1016/j.aim.2016.07.028}.
\newblock URL \url{https://arxiv.org/abs/1503.06469}.

\bibitem[Clementino and Lucatelli~Nunes(2024)]{ClementinoLucatelliNunes2024}
Maria~Manuel Clementino and Fernando Lucatelli~Nunes.
\newblock Lax comma 2-categories and admissible 2-functors.
\newblock \emph{Theory Appl. Categ.}, 40\penalty0 (6):\penalty0 180--226, 2024.
\newblock \doi{10.70930/tac/re3rk47b}.
\newblock URL \url{https://www.tac.mta.ca/tac/volumes/40/6/40-06abs.html}.

\bibitem[Emmenegger et~al.(2024)Emmenegger, Mesiti, Rosolini, and
  Streicher]{EmmeneggerEtAl2024}
Jacopo Emmenegger, Luca Mesiti, Giuseppe Rosolini, and Thomas Streicher.
\newblock A comonad for {Grothendieck} fibrations.
\newblock \emph{Theory Appl. Categ.}, 40\penalty0 (13):\penalty0 371--389,
  2024.
\newblock \doi{10.70930/tac/3d9q6q64}.

\bibitem[Fauser and Vickers(2014)]{FauserVickers2014}
Bertfried Fauser and Steven Vickers.
\newblock Geometric constructions preserve fibrations.
\newblock arXiv:1411.2457 [math.CT], 2014.
\newblock URL \url{https://arxiv.org/abs/1411.2457}.

\bibitem[Grandis and Tholen(2006)]{GrandisTholen2006}
Marco Grandis and Walter Tholen.
\newblock Natural weak factorization systems.
\newblock \emph{Arch. Math. (Brno)}, 42\penalty0 (4):\penalty0 397--408, 2006.
\newblock URL \url{https://eudml.org/doc/249802}.

\bibitem[Gray(1966)]{Gray1966}
John~W. Gray.
\newblock Fibred and cofibred categories.
\newblock In \emph{Proceedings of the Conference on Categorical Algebra, La
  Jolla, 1965}, pages 21--83. Springer, New York, 1966.
\newblock \doi{10.1007/978-3-642-99902-4_2}.

\bibitem[Grothendieck(1971)]{Grothendieck1971}
Alexander Grothendieck.
\newblock Cat\'egories fibr\'ees et descente.
\newblock In \emph{Rev\^etements \`etales et groupe fondamental ({SGA} 1)},
  volume 224 of \emph{Lecture Notes in Mathematics}, pages 145--194. Springer,
  Berlin, 1971.
\newblock \doi{10.1007/BFb0058662}.

\bibitem[Hermida(1999)]{Hermida1999}
Claudio Hermida.
\newblock Some properties of {Fib} as a fibred 2-category.
\newblock \emph{J. Pure Appl. Algebra}, 134\penalty0 (1):\penalty0 83--109,
  1999.
\newblock \doi{10.1016/S0022-4049(97)00129-1}.

\bibitem[Jacobs(1999)]{Jacobs1999}
Bart Jacobs.
\newblock \emph{Categorical Logic and Type Theory}, volume 141 of \emph{Studies
  in Logic and the Foundations of Mathematics}.
\newblock North-Holland, Amsterdam, 1999.
\newblock ISBN 0-444-50170-3.

\bibitem[Janelidze and Tholen(1999)]{JanelidzeTholen1999}
George Janelidze and Walter Tholen.
\newblock Functorial factorization, well-pointedness and separability.
\newblock \emph{J. Pure Appl. Algebra}, 142\penalty0 (2):\penalty0 99--130,
  1999.
\newblock \doi{10.1016/S0022-4049(98)00095-4}.

\bibitem[Kelly and Street(1974)]{KellyStreet1974}
G.~Max Kelly and Ross Street.
\newblock Review of the elements of 2-categories.
\newblock In G.~Max Kelly, editor, \emph{Category Seminar, Sydney 1972/1973},
  volume 420 of \emph{Lecture Notes in Mathematics}, pages 75--103. Springer,
  Berlin, 1974.
\newblock \doi{10.1007/BFb0063101}.

\bibitem[Kock(1995)]{Kock1995}
Anders Kock.
\newblock Monads for which structures are adjoint to units.
\newblock \emph{J. Pure Appl. Algebra}, 104\penalty0 (1):\penalty0 41--59,
  1995.
\newblock \doi{10.1016/0022-4049(94)00111-U}.

\bibitem[Korostenski and Tholen(1993)]{KorostenskiTholen1993}
Mareli Korostenski and Walter Tholen.
\newblock Factorization systems as {Eilenberg--Moore} algebras.
\newblock \emph{J. Pure Appl. Algebra}, 85\penalty0 (1):\penalty0 57--72, 1993.
\newblock \doi{10.1016/0022-4049(93)90171-O}.

\bibitem[Lucatelli~Nunes(2019)]{LucatelliNunes2019}
Fernando Lucatelli~Nunes.
\newblock Pseudoalgebras and non-canonical isomorphisms.
\newblock \emph{Appl. Categ. Structures}, 27\penalty0 (1):\penalty0 55--63,
  2019.
\newblock \doi{10.1007/s10485-018-9541-3}.

\bibitem[Lucatelli~Nunes and V{\'a}k{\'a}r(2023)]{LucatelliNunesVakar2023}
Fernando Lucatelli~Nunes and Matthijs V{\'a}k{\'a}r.
\newblock {CHAD} for expressive total languages.
\newblock \emph{Math. Structures Comput. Sci.}, 33\penalty0 (4--5):\penalty0
  311--426, 2023.
\newblock \doi{10.1017/S096012952300018X}.
\newblock URL \url{https://arxiv.org/abs/2110.00446}.

\bibitem[Lucatelli~Nunes and V{\'a}k{\'a}r(2025)]{LucatelliNunesVakar2025}
Fernando Lucatelli~Nunes and Matthijs V{\'a}k{\'a}r.
\newblock Monoidal closure of {Grothendieck} constructions via
  {$\Sigma$}-tractable monoidal structures and {Dialectica} formulas.
\newblock \emph{Theory Appl. Categ.}, 44\penalty0 (35):\penalty0 1153--1217,
  2025.
\newblock URL \url{https://www.tac.mta.ca/tac/volumes/44/35/44-35abs.html}.

\bibitem[Lucatelli~Nunes et~al.(2023)Lucatelli~Nunes, Prezado, and
  Sousa]{LucatelliNunesPrezadoSousa2023}
Fernando Lucatelli~Nunes, Rui Prezado, and Lurdes Sousa.
\newblock Cauchy completeness, lax epimorphisms and effective descent for split
  fibrations.
\newblock \emph{Bull. Belg. Math. Soc. Simon Stevin}, 30\penalty0 (1):\penalty0
  130--139, 2023.
\newblock \doi{10.36045/j.bbms.221021}.
\newblock URL \url{https://arxiv.org/abs/2210.12021}.

\bibitem[Moeller and Vasilakopoulou(2020)]{MoellerVasilakopoulou2020}
Joe Moeller and Christina Vasilakopoulou.
\newblock Monoidal {Grothendieck} construction.
\newblock \emph{Theory Appl. Categ.}, 35\penalty0 (31):\penalty0 1159--1207,
  2020.
\newblock URL \url{https://www.tac.mta.ca/tac/volumes/35/31/35-31abs.html}.

\bibitem[Perrone and Tholen(2022)]{PerroneTholen2022}
Paolo Perrone and Walter Tholen.
\newblock Kan extensions are partial colimits.
\newblock \emph{Appl. Categ. Structures}, 30\penalty0 (4):\penalty0 685--753,
  2022.
\newblock \doi{10.1007/s10485-021-09671-9}.

\bibitem[Peschke and Tholen(2020)]{PeschkeTholen2020}
George Peschke and Walter Tholen.
\newblock Diagrams, fibrations, and the decomposition of colimits.
\newblock arXiv:2006.10890 [math.CT], 2020.
\newblock URL \url{https://arxiv.org/abs/2006.10890}.

\bibitem[Rosebrugh and Wood(2002{\natexlab{a}})]{RosebrughWood2002}
Robert Rosebrugh and R.~J. Wood.
\newblock Coherence for factorization algebras.
\newblock \emph{Theory Appl. Categ.}, 10\penalty0 (6):\penalty0 134--147,
  2002{\natexlab{a}}.
\newblock \doi{10.70930/tac/g0lqukuu}.

\bibitem[Rosebrugh and Wood(2002{\natexlab{b}})]{RosebrughWood2002Distributive}
Robert Rosebrugh and R.~J. Wood.
\newblock Distributive laws and factorization.
\newblock \emph{J. Pure Appl. Algebra}, 175\penalty0 (1--3):\penalty0 327--353,
  2002{\natexlab{b}}.
\newblock \doi{10.1016/S0022-4049(02)00140-8}.

\bibitem[Rosick\'{y} and Tholen(2002)]{RosickyTholen2002}
Ji\v{r}\'{\i} Rosick\'{y} and Walter Tholen.
\newblock Lax factorization algebras.
\newblock \emph{J. Pure Appl. Algebra}, 175\penalty0 (1--3):\penalty0 355--382,
  2002.
\newblock \doi{10.1016/S0022-4049(02)00141-X}.

\bibitem[Rosick\'{y} and Tholen(2007)]{RosickyTholen2007}
Ji\v{r}\'{\i} Rosick\'{y} and Walter Tholen.
\newblock Factorization, fibration and torsion.
\newblock \emph{J. Homotopy Relat. Struct.}, 2\penalty0 (2):\penalty0 295--314,
  2007.
\newblock URL \url{https://eudml.org/doc/230964}.

\bibitem[Street(1972)]{Street1972}
Ross Street.
\newblock The formal theory of monads.
\newblock \emph{J. Pure Appl. Algebra}, 2\penalty0 (2):\penalty0 149--168,
  1972.
\newblock \doi{10.1016/0022-4049(72)90019-9}.

\bibitem[Street(1974)]{Street1974}
Ross Street.
\newblock Fibrations and {Yoneda}'s lemma in a 2-category.
\newblock In G.~Max Kelly, editor, \emph{Category Seminar, Sydney 1972/1973},
  volume 420 of \emph{Lecture Notes in Mathematics}, pages 104--133. Springer,
  Berlin, 1974.
\newblock \doi{10.1007/BFb0063102}.

\bibitem[Street(1980)]{Street1980}
Ross Street.
\newblock Fibrations in bicategories.
\newblock \emph{Cah. Topol. G\'{e}om. Diff\'{e}r. Cat\'{e}g.}, 21\penalty0
  (2):\penalty0 111--160, 1980.
\newblock URL \url{https://www.numdam.org/item/CTGDC_1980__21_2_111_0/}.

\bibitem[Street and Walters(1973)]{StreetWalters1973}
Ross Street and R.~F.~C. Walters.
\newblock The comprehensive factorization of a functor.
\newblock \emph{Bull. Amer. Math. Soc.}, 79\penalty0 (5):\penalty0 936--941,
  1973.
\newblock \doi{10.1090/S0002-9904-1973-13268-9}.

\bibitem[Vasilakopoulou(2018)]{Vasilakopoulou2018}
Christina Vasilakopoulou.
\newblock On enriched fibrations.
\newblock \emph{Cah. Topol. G\'{e}om. Diff\'{e}r. Cat\'{e}g.}, 59\penalty0
  (4):\penalty0 354--387, 2018.
\newblock URL
  \url{https://cahierstgdc.com/wp-content/uploads/2018/10/Vasilakopoulou-LIX-4.pdf}.

\bibitem[Z{\"o}berlein(1976)]{Zoberlein1976}
Volker Z{\"o}berlein.
\newblock Doctrines on 2-categories.
\newblock \emph{Math. Z.}, 148\penalty0 (3):\penalty0 267--279, 1976.
\newblock \doi{10.1007/BF01214522}.

\end{thebibliography}

\end{document}